\documentclass[a4paper,12pt]{article}

\usepackage[english]{babel}
\usepackage{amsmath}
\usepackage{amsthm}
\usepackage{enumerate}
\usepackage{graphicx}
\usepackage{mathrsfs}
\usepackage{geometry}
\usepackage{microtype}

\newtheorem{thm}{Theorem}[section]
\newtheorem{remark}[thm]{Remark}
\newtheorem{defn}[thm]{Definition}
\newtheorem{lemma}[thm]{Lemma}

\numberwithin{equation}{section}

\title{Time relaxation of a phase-field model with entropy balance}
\author{Manuela Girotti \footnote{Department of Mathematics and Statistics, Concordia University,  1455 de Maisonneuve West, H3G 1M8, Montr\'eal, Qu\'ebec, Canada; e-mail: mgirotti@mathstat.concordia.ca}}
\date{}

\begin{document}
\maketitle

\begin{abstract}
We deal with a system of two coupled differential equations, describing the evolution of a first order phase transition. In particular, we have two non-linear parabolic equations: the first one is deduced from a balance law for entropy and it describes the evolution of the absolute temperature; the other one is an equilibrium equation for microforces and it regulates the behaviour of a scalar phase parameter. Moreover, the second equation shows a time-relaxation coefficient related to the time-derivative of the phase parameter. 

We prove well-posedness of solution to the given system, using a standard method of approximating problems; afterwards, we study the behaviour of the system as the time-relaxation coefficient tends to zero: as a result, we find out that the original problem converges to a new problem, with a stationary phase equation. 
\end{abstract}

\section{Introduction}
This paper deals with a system of two coupled differential equations, describing the evolution of a first order phase transition. Both equations are non-linear and parabolic: the first one is deduced from a balance law for entropy and it describes the evolution of the absolute temperature $\vartheta$; the other one is an equilibrium equation for microforces, which are responsible for the phase transition process, and it regulates the behavior of a scalar phase parameter $\chi$. 

The system is regarded in the parabolic cylinder $Q := \Omega \times (0,T)$, with $T>0$ and $\Omega \subseteq \textbf{R}^3$, smooth and bounded, and it has following expression:
\begin{equation}
\begin{split}
&\partial_t (\log \vartheta + \chi) - \Delta \vartheta = g  \\
&\mu \chi_t - \Delta \chi + \xi + \sigma'(\chi) =  \vartheta  \ \ \ \text{and} \ \ \ \xi \in \partial \beta (\chi).
\end{split} \label{intro1}
\end{equation}
where $g$ is a thermal source, $\partial \beta$ is the subdifferential of an appropriate proper, convex and lower semi-continuous function $\beta$ ($\partial \beta$ is indeed a maximal monotone graph) and $\sigma$ is a $C^1$ function with Lipschitz continuous derivative. 

The sum $\partial \beta + \sigma'$ introduces a non-linear term related to some physical constraints (thermodynamically consistent) which can eventually occur in applications. 

Finally, the coefficient $\mu$ represents a time-relaxation parameter and it is usually a small quantity in applications, compared with other physical quantities which act during phase-transitions. 

The PDE system is completed with initial and boundary data: a non-homogeneous Robin condition is provided for the absolute temperature $\vartheta$  and a standard Neumann condition is set for the order parameter $\chi$.
\begin{equation}
\begin{array}{ll}
\partial _\nu \vartheta + \alpha \vartheta = h &\text{on} \ \partial \Omega \times (0, T)\\
\partial_{\nu} \chi = 0 &\text{on} \ \partial \Omega \times (0, T) \\
\log \vartheta (0) = \log \vartheta_0  &\text{and} \ \ \ \chi (0) = \chi_0 \ \ \ \text{in} \ \Omega,
\end{array} 
\label{intro2}
\end{equation}
where $\alpha$ is an arbitrary positive function, with appropriate regularity, and $h$ denotes the boundary thermal supply; $\partial _\nu$ indicates the outward normal derivative. 
The initial data are given by two function $\vartheta_0$ and $\chi_0$ defined on $\Omega$ and in particular we set $\vartheta_0 >0$ on $\Omega$.

The main difficulties in dealing with this system are the non-linear terms, in particular the logarithmic term and the maximal monotone graph $\partial \beta$. However, the presence of the logarithm enable us to conclude straightforward the positiveness of the temperature $\vartheta$, once the existence of solution has been proved in some suitable functional space. 

To obtain the result of existence of solutions, we have to go through a double approximation of the problem: first we regularize the non-linear terms and the initial data and subsequently we proceed with a Faedo-Galerkin procedure to solve the approximated system. Finally we attain to the existence by using appropriate compactness results. 
 
After that, we are able to prove the uniqueness of solution to our problem. 

The second target of this paper is the study of the behavior of the system as the coefficient $\mu$ tends to $0$. In particular, we will find out that the original problem $(P_\mu)$, with $\mu >0$, converges to a problem $(P_0)$, where the time derivative of $\chi$ does not appear in the phase equation. 

The study of this asymptotic behavior has remarkable interest in many physical situations, where usually the $\mu$ parameter is smaller with respect to other physical quantities (such as the interfacial energy coefficient, which is related to the laplacian term in the phase equation). 

Since we have a stationary equation for the phase parameter, as $\mu \searrow 0$, we loose time regularity properties for $\chi$ and consequently for the logarithmic term, too. Thus, we need to impose further hypotheses on the non-linear terms $\beta$ and $\sigma$ and to introduce a generalization of the function $\log \vartheta$ in the space $(H^1(\Omega))'$, in order to prove the existence of solution to the limit problem $(P_0)$.

This system was first introduced and studied by Bonetti \textit{et al.} in \cite{model} and \cite{global}. In those articles the mathematical model of the system above (with extra terms involving a thermal memory kernel) is explicitly derived and existence and uniqueness of solution is proved with $\mu >0$ fixed and with Dirichlet boundary condition for the variable $\vartheta$. 
Moreover, the latent heat is thought as an arbitrary function of $\chi$,not constant, as we assume in this paper. 

On the other hand, the use of Robin boundary condition in this paper introduce  new difficulties in our system in proving the existence of solutions, since we have to deal with boundary terms which have to be studied using appropriate trace theorems. 

A similar model, which shows an entropy balance equation for the evolution of the temperature $\vartheta$, instead of the usual energy balance equation, has been studied by Bonetti, Colli and Fr\'emond in \cite{model2} and by Bonetti and Fr\'emond in \cite{model3}. In particular, both articles show explicit forms of non-linearities, while in the present article the non-linear term is represented by generic functions $\beta$ and $\sigma$, and the logarithmic term appears also under the laplacian; moreover, the article \cite{model2} shows a thermal memory kernel.

Regarding the study of asymptotic behavior of the solutions as a certain coefficient tends to zero, we can quote the article of Gilardi and Rocca \cite{convergence}, where, starting from the same model proposed in \cite{global} and \cite{model}, the convergence to $0$ of the energy interface coefficient in front of $\Delta \chi$ in the phase parameter equation is analyzed. 

We quote also the article of Colli \textit{et al.} \cite{hiroshima}, where the asymptotic behavior as the time relaxation parameter $\mu \searrow 0$ is analyzed, but we point out that in this case a Caginalp phase-field model with memory is taken into exam; moreover, the equation regarding the behavior of the temperature $\vartheta$ is derived from an energy balance (and not from an entropy balance as in this article) and does not show a logarithmic term.

The paper is organized as follows: in Section \ref{model} we provide a description of the mathematical model which is later studied in Section \ref{exist}, where the well-posedness of the problem is proved. 
Finally, in the last Section \ref{mu} we study the asymptotic behavior of the problem as the time relaxation parameter $\mu$ tends to $0$. 
All the results of the present paper are stated in Section \ref{result}.

\section{The model}
\label{model}

Consider a first-order phase transition occurring in a smooth domain of the physical space $ \Omega \subseteq \textbf{R}^3$. The unknown variables of our problem are the absolute temperature $\vartheta \in (0, + \infty)$ and the order parameter $\chi$. In particular, $\chi$ represents a local concentration or the rescaled proportion of one phase with respect to the other and it is related to the microscopical movements of particles. 

In order to grant thermomechanical consistency to the model, we need to introduce some constraints on $\chi$, which will be derived from suitable functions $\beta$ and $\sigma$, as we will see later. 

The derivation of the model is mainly taken from the article by Bonetti \textit{et al.} (\cite{model}), which is based on an approach independently proposed by Gurtin (\cite{gurtin}) and Fr\'emond (\cite{fremond}). We repeat it here for the reader's convenience.

Assuming that macroscopic deformations do not occur during the process and that microscopic accelerations are negligible, we start stating the following balance law for microforces (compare with \cite{fremond}):
\begin{equation}
\begin{array}{ll}
\text{div} \textbf{H} + b = \textbf{B} &\text{in} \ \Omega\\
\textbf{H} \cdot \nu = 0 &\text{on} \ \partial \Omega, \label{Hboundary}
\end{array}
\end{equation}
where $\textbf{H}$ and $\textbf{B}$ are interior microscopic forces and $b$ is an external force acting on the body at microscopic level; in particular, we are assuming that no external contact force acts on the boundary of our domain $\Omega$: this is the reason why we imposed a no-flux boundary condition for $\textbf{H}$. 
We can think of all these microforces as mechanically induced heat sources, which have to be taken into account while dealing with the first law of thermodynamics (in a local form):
\begin{equation}
e_t = - \text{div} \textbf{q} + r + \text{div}(\textbf{H} \chi_t) + b \chi_t  = - \text{div} \textbf{q} + r + \textbf{H} \cdot \nabla \chi_t + B \chi_t,
\label{1termodin}
\end{equation}
with $e$ the internal energy of the system, $\textbf{q}$ the heat flux and $r$ an additional heat source. 

At this point, we aim to study the thermomechanical consistency of the model and to find explicit expressions for the microscopic forces $\textbf{B}$ and $\textbf{H}$. First of all, let recall the second law of thermodynamics in the form of the Clausius-Duhem inequality
\begin{equation}
\eta_t \geq - \text{div} \textbf{Q} + g,
\label{2termodin}
\end{equation}
with $\eta$ the entropy of the system, $\textbf{Q} := \textbf{q}/\vartheta$ the entropy flux and $g:= r/\vartheta$ the external entropy supply; we introduce also the free-energy functional $\Psi$ (see \cite{model2}, \cite{brokate}, \cite{caginalp} e \cite{schimperna})
\begin{equation}
\Psi (\vartheta, \chi) = \mathscr{F}(\vartheta, \chi) + \frac{\varepsilon}{2} \left| \nabla \chi \right|^2,
\end{equation}
with $\mathscr{F}$ the density of the free energy for pure phases; $|\nabla \chi|^2$ takes into account local interactions between phases, while the constant $\varepsilon >0$ is the energy interface coefficient (see \cite[Sec. 3 and 4]{caginalpfife} and \cite[Sec. 3]{fix} for a thorough discussion on the argument).

Using the Helmholtz's relation
\begin{equation}
\eta = - \frac{\partial \Psi}{\partial \vartheta},  \label{helmholtz}
\end{equation}
and assuming that the entropy has the following expression
\begin{equation}
\eta = c_s \left(1 + \log \vartheta \right) + \ell \chi, \label{entropy}
\end{equation}
where $c_s \in \textbf{R}_+$ is the specific heat and $\ell \in \textbf{R}_+$ is the latent heat associated with the phase transition, we have
\begin{equation}
\Psi (\vartheta, \chi) = - c_s \vartheta \log \vartheta - \ell \vartheta \chi + \left[ \beta + \sigma \right](\chi) + \frac{\varepsilon}{2} \left| \nabla \chi \right|^2, \label{pottermo}
\end{equation}
with $\beta: \textbf{R} \rightarrow [0,+ \infty]$ a proper, convex and lower semi-continuous function and $\sigma: \textbf{R} \rightarrow \textbf{R}$ another (suitably regular) function. 
The functions we have right now introduced represent arbitrary constraints on the order parameter $\chi$; by suitably choosing $\beta$ and $\sigma$, the model may describe different types of phase transitions. 

\paragraph{Examples of possible non-linearities.}{
The non-linear term $\beta + \sigma$ may have different significant expressions. 

If we are assuming that the two phases may coexist with different proportions (i.e. we are assuming the possibility of the so called ``mushy regions''), we can reasonably impose the following constraint
\begin{equation}
\chi \in [0,1]
\end{equation}
setting $1-\chi$ equal to the proportion of the other phase; in particular, the values $\chi =0,1$ correspond to the pure phases, while for the in-between values both phases coexist at each point of the body.

To force the order parameter to take only the required values, we can set $\beta = I_{[0,1]}$ the indicator function of the interval $[0,1]$, defined by $I_{[0,1]}(\chi) = 0$ if $\chi \in [0,1]$ and $I_{[0,1]}(\chi) = + \infty$ elsewhere. The subdifferential of $\beta$, which will appear in the phase equation, is then the following maximal monotone graph
\begin{equation}
\xi \in \partial I_{[0,1]}(\chi) \ \ \ \Longleftrightarrow \ \ \  \left\{ 
\begin{array}{ll}
\xi \leq 0 & \text{if} \ \chi=0 \\
\xi =0 & \text{if} \ \chi \in (0,1)\\
\xi \geq 0 & \text{if} \ \chi =1.
\end{array}
\right.
\end{equation}
In this case, we can choose the following expression for the function $\sigma$, which is quite common for solid-liquid phase transitions,
\begin{equation}
\sigma (\chi) = -\ell \vartheta_c + 4 M \chi(1 - \chi),
\end{equation}
where $\vartheta_c$ is the critical temperature of the transition and $M$ is the maximum value of $\sigma$, attained at $\chi = 1/2$ (see \cite[Sec. 2.4]{model}).

Another important model is the so called ``double-well" model, obtained by choosing the following form for the potential
\begin{equation}
W = \beta + \sigma = \frac{1}{4}\left( \chi^2 - 1 \right)^2,
\end{equation}
which has two minima at the points $\chi = \pm 1$ (pure phases) and a maximum at the point $\chi = 0$ (transition point).

In this case, the non linear term in the phase equation will be $\beta' (\chi) + \sigma'(\chi) = \chi^3 - \chi$.
}

\paragraph{}
Combining the first and second law of thermodynamics ((\ref{1termodin}) and (\ref{2termodin})), we get the inequality
\begin{align}
\Psi_t &= e_t - \vartheta_t \eta - \vartheta \eta_t \nonumber \\
&\leq \left[ - \text{div} \textbf{q} + r + \textbf{H} \cdot \nabla \chi_t + B \chi_t \right] - \vartheta_t \eta - \vartheta \left[ - \text{div} \left( \frac{\textbf{q}}{\vartheta} \right) + \frac{r}{\vartheta} \right]  \nonumber \\
&= - \vartheta_t \eta - \frac{1}{\vartheta} \textbf{q} \cdot \nabla \vartheta + B \chi_t + \textbf{H} \cdot \nabla \chi_t,
\label{3termodin}
\end{align}
which has to be identically fulfilled by any admissible process $\mathscr{P} = ( \vartheta_t , \chi_t , \nabla \chi_t , \nabla \vartheta)$. The first equality follows from the well known relation: $e = \Psi + \vartheta \eta$. 

Let consider the processes $\mathscr{P} = (0, \chi_t, \textbf{0}, \textbf{0})$ and $\mathscr{P} = (0, 0, \nabla \chi_t, \textbf{0})$ and apply them to the equation (\ref{3termodin}):
\begin{equation}
\begin{split}
\frac{\partial \Psi}{\partial \chi} \chi_t - B \chi_t &\leq 0  \\
\left( \frac{\partial \Psi}{\partial (\nabla \chi)} - \textbf{H} \right) \cdot \nabla \chi_t &\leq 0, \label{relBH}
\end{split}
\end{equation}
for each choice of $\chi_t$ and $\nabla \chi_t$. 

Under the hypothesis of small perturbations, we can assume the functionals $\textbf{B}$ and $\textbf{H}$ to be linearly depended on the dissipative variables:
\begin{equation}
B = B_{nd} + B_d \chi_t \ \ \ \text{and} \ \ \ \textbf{H} = \textbf{H}_{nd} + H_d \nabla \chi_t.
\end{equation}

Combining these two expressions together with the previous inequalities (\ref{relBH}), we obtain the following state laws
\begin{equation}
\begin{split}
&B_{nd} = \frac{\partial \Psi}{\partial \chi} = - \ell \vartheta + \xi + \sigma'(\chi), \ \ \ B_d \geq 0, \ \ \ \text{where} \ \xi \in \partial \beta(\chi),  \\
&\textbf{H}_{nd} = \frac{\partial \Psi}{\partial (\nabla \chi)} = \varepsilon \nabla \chi, \ \ \ H_d \geq 0. \label{relB1H1}
\end{split}
\end{equation}
Moreover, we set 
\begin{equation}
B_d = \mu \geq 0 \ \ \ \text{constant}, \ \ \ H_d = 0. \label{relB2H2}
\end{equation}

Furthermore, recalling the general formula for heat conductors $ \textbf{q} = - k(\vartheta) \nabla \vartheta$, we set the heat conductivity $k(\vartheta)$ to be a linear function of the temperature: 
\begin{equation} 
\textbf{q} = - k_0 \vartheta \nabla \vartheta, \ \ \ \text{with} \ k_0 >0;
\end{equation}
this choice is quite standard for many dielectrics, like ice or water.

At this point, we can rewrite the first law of thermodynamics (\ref{1termodin}) in the following way
\begin{align}
e_t &= - \text{div} \textbf{q} + r + (B_{nd} + B_d \chi_t)\chi_t + \textbf{H}_{nd} \cdot \nabla \chi_t   \nonumber \\
&= - \vartheta \text{div} \left( \frac{\textbf{q}}{\vartheta} \right) - \frac{1}{\vartheta} \textbf{q} \cdot \nabla \vartheta + r + B_{nd} \chi_t + B_d \left|\chi_t \right|^2 + \textbf{H}_{nd} \cdot \nabla \chi_t.
\end{align}
On the other hand, thank's to relations (\ref{relB1H1}), 
\begin{align}
e_t &= \Psi_t + \vartheta_t \eta + \vartheta \eta_t = \left[ \frac{\partial \Psi}{\partial \vartheta}\vartheta_t + \frac{\partial \Psi}{\partial \chi}\chi_t + \frac{\partial \Psi}{\partial (\nabla \chi)}\nabla \chi_t \right] + \vartheta_t \eta + \vartheta \eta_t  \nonumber \\
&= B_{nd} \chi_t + \textbf{H}_{nd} \cdot \nabla \chi_t + \vartheta \eta_t.
\end{align}
Then, by comparison, we have: $\vartheta \eta_t = - \vartheta \text{div} \left(- k_0 \nabla \vartheta \right) - \frac{1}{\vartheta} \textbf{q} \cdot \nabla \vartheta + r + \mu \left|\chi_t \right|^2$.

Now we are allowed to neglect the high order non linearities, thank's to the assumption of small perturbations, divide the equation by $\vartheta >0$ and explicit the entropy $\eta$:
\begin{equation}
\partial_t (c_s \log \vartheta + \ell \chi) - k_0 \Delta \vartheta = g, \label{eq1modello}
\end{equation}
where $g:= r/\vartheta$.

Finally, assuming that no external microscopic force acts on the body ($b=0$ in equation (\ref{Hboundary})), thank's to relations (\ref{relB1H1}) and (\ref{relB2H2}), the microforces balance equations turns out to be the following parabolic non linear equation
\begin{equation}
\varepsilon \Delta \chi = - \ell \vartheta + \xi + \sigma'(\chi) + \mu \chi_t,
\end{equation}
with $\xi \in \partial \beta (\chi)$.

The above equations are combined with suitable boundary and initial conditions. In particular, concerning boundary conditions and recalling (\ref{Hboundary}), we fix a homogeneous Neumann condition for $\chi$:
\begin{equation}
\textbf{H} \cdot \nu = \varepsilon \partial_{\nu} \chi = 0 \ \ \ \text{on} \ \partial \Omega \times (0, T),
\end{equation}
with $\nu$ the outer normal derivative on the boundary of the domain; while, for the absolute temperature, we fix a Robin condition 
\begin{equation}
\partial_\nu \vartheta + \alpha \vartheta = h \ \ \ \text{on} \ \partial \Omega \times (0, T), 
\end{equation}
with $\alpha$ a suitable real-valued function defined on $\partial \Omega$ and $h$ a boundary heat source.

The initial data are the following
\begin{equation}
\log \vartheta (0) = \log \vartheta_0 \ \ \ \text{and} \ \ \ \chi (0) = \chi_0.
\end{equation}

\section{Main results}
\label{result}

In this section we show the rigorous formulation of our problem (\ref{intro1})-(\ref{intro2}) and we state our main results. 

\subsection{Well-posedness of problem $(P_\mu)$, with $\mu>0$}
Let consider a bounded, connected domain $\Omega \subseteq \textbf{R}^3$, with suitably regular boundary $\Gamma := \partial \Omega$ . Let $T >0$ be the final time and $Q := \Omega \times (0,T)$ and $\Sigma := \Gamma \times (0,T)$ be the parabolic cylinder and the parabolic boundary, respectively. 

We set the following Hilbert triplet (see \cite{brezisfunz})
\begin{equation}
(V,H,V')=(H^1(\Omega), L^2(\Omega), (H^1(\Omega))' )
\end{equation}
with the usual standard norms and inner products. 

Let define the following linear, continuous operators 
\begin{equation}
\begin{split}
&B: V \rightarrow V', \ \ \ 
\langle Bu, v \rangle = \int_\Omega{\nabla u \cdot \nabla v \, dx} + \int_{\Gamma}{\alpha \, u \,v \,ds} \ \ \ \forall \, u,v \in V; \\
&A: V \rightarrow V', \ \ \ 
\langle Au, v \rangle = \int_\Omega{\nabla u \cdot \nabla v \, dx} \ \ \ \forall \, u,v \in V; 
\end{split}
\end{equation}
with $\alpha \in L^{\infty}(\Gamma)$:
\begin{equation}
0 < \overline{\alpha} \leq \alpha(x) \leq \widehat{\alpha}  \ \ \ \text{a.e. on} \ \Gamma, \ \text{for some} \ \overline{\alpha}, \widehat{\alpha} \in \textbf{R}_+. \label{alpha}
\end{equation}

\begin{remark}
The operator $B$ is symmetric and coercive, thus it defines a scalar product in $V$, which is equivalent to the standard one (thanks to Poincar\'e's inequality). From now on, we refer to $||\cdot||_V$ for the norm induced by $B$.
\end{remark}

The hypotheses on the source terms are the following ones
\begin{equation}
\begin{split}
&g \in C^0([0,T]; H),  \\
&h \in L^2(0,T;L^{2}(\Gamma)) \cap W^{1,1}(0,T; L^{2}(\Gamma)), \ \ \ h \geq 0 \ \ \ \text{a.e. in} \ \Sigma; \label{datogh}
\end{split}
\end{equation}
moreover, we define the operator $w \in C^0([0,T];V')$ as
\begin{equation}
\langle w(t), v \rangle = \int_\Omega{g (t)\, v dx} + \int_{\Gamma}{h (t) \,v ds} \ \ \ \forall \, t \in (0,T), \ \forall \, v \in V. \label{defw}
\end{equation}

The constraints on the phase parameter are given by the following functions 
\begin{equation}
\begin{split}
&\beta: \textbf{R} \rightarrow [0, + \infty] \ \ \text{is a proper, convex, lower semi-continuous function} \\
&\text{such that} \ \beta(0) = 0;   \\
&\sigma \in C^1(\textbf{R}), \ \ \sigma' \ \text{is Lipschitz continuous, with Lipschitz constant} \ c_L. \label{betasigma}
\end{split}
\end{equation}
Hence, we consider the subdifferential $\partial \beta$ of the function $\beta$ which turns out to be a maximal monotone graph, with $\partial \beta(0) \ni 0$ (see \cite[Chap. II]{brezis}). 

\begin{remark}
Thanks to (\ref{betasigma}), we deduce the following polynomial growth for $\sigma$
\begin{equation}
\left| \sigma (r) \right| \leq c_\sigma (1 + r^2) \ \ \ \forall \, r \in \textbf{R}. \label{stimasigma}
\end{equation}
\end{remark}

Finally, the initial data has the following regularity properties
\begin{equation}
\begin{split}
&\chi_{0,\mu} \in V, \ \ \beta(\chi_{0,\mu}) \in L^1(\Omega)  \\
&\vartheta_0 \in L^{\infty}(\Omega), \ \  \vartheta_0 > 0 \ \ \text{a.e. in} \ \Omega \ \text{and} \ 1/\vartheta_0 \in L^{\infty}(\Omega). \label{chizerothetabound}
\end{split}
\end{equation}
Moreover, these conditions imply $0 < \vartheta_* \leq \vartheta_0 (x) \leq \vartheta^*$, for a.e. $x \in \Omega$, for some $\vartheta_*, \vartheta^* \in \textbf{R}_+$.

In conclusion, our problem $(P_\mu)$, with $\mu >0$ fixed, turns out to be the following: we search for a pair $(\vartheta_\mu, \chi_\mu)$ satisfying
\begin{align}
&\partial_t \left( \log \vartheta_\mu  + \chi_\mu \right) + B\vartheta_\mu = w \ \ \ \text{in} \ V' \ \text{and a.e. in} \ (0,T) \label{1eqPmu}\\
&\mu \partial_t \chi_\mu + A\chi_\mu + \xi_\mu + \sigma'(\chi_\mu) = \vartheta_\mu \ \ \ \text{a.e. in} \ Q  \label{2eqPmu}\\
&\xi_\mu \in \partial \beta(\chi_\mu) \ \ \ \text{a.e. in} \ Q\\
&\log \vartheta_\mu(0) = \log \vartheta_0 \ \ \ \text{and} \ \ \ \chi_\mu(0) = \chi_{0,\mu} \ \ \ \text{a.e. in} \ \Omega. \label{ultimaPmu}
\end{align}
\begin{remark} 
In this formulation of the problem $(P_\mu)$, one can notice that some coefficients are missing, in particular the energy interface coefficient $\varepsilon$, the latent heat $\ell$, the specific heat $c_s$ and the heat conductivity $k_0$.
For the sake of simplicity,  these physical constants had been normalized to $1$.
\end{remark}

Moreover, thanks to the properties of the functionals $A$ and $B$, the following functional spaces are well-defined
\begin{equation}
\begin{split}
& D(A;H) := A^{-1}(H) = \left\{ u \in H^2(\Omega) \left| \, \partial_\nu u = 0 \right. \right\} \\
& D(B;H) := B^{-1}(H) = \left\{ u \in V \left| \, \Delta u \in H, \ \partial_{\nu}u + \alpha u = h \right. \right\}.
\end{split}
\end{equation}

\begin{thm}
\label{teo1}
Let $\mu >0$ and $T>0$ fixed. Under the assumption (\ref{alpha}), (\ref{datogh}), (\ref{betasigma}), (\ref{chizerothetabound}), there exists a unique pair $(\vartheta_\mu, \chi_\mu)$ and there exists a selection $\xi_\mu$ such that
\begin{equation}
\begin{split}
&\vartheta_\mu \in L^2(0,T; V), \ \ \ \vartheta_\mu > 0 \ \ \ \text{a.e. in} \ Q  \\
&\log \vartheta_\mu \in L^{\infty}(0,T;H) \cap H^1(0,T; V')  \\
&\chi_\mu \in L^2(0,T; D(A;H)) \cap H^1(0,T; H)  \\
&\xi_\mu \in L^2(Q)
\end{split}
\end{equation}
satisfying the problem (\ref{1eqPmu})-(\ref{ultimaPmu}).

If the norms of all data, related to (\ref{datogh}) and (\ref{chizerothetabound}), are bounded by a positive constant $M$, then the solution $(\vartheta_\mu, \chi_\mu, \xi_\mu)$ satisfies the following estimate
\begin{gather}
\left\| \vartheta_\mu \right\|_{L^2(0,T;V)} + \left\| \log \vartheta_\mu \right\|_{L^{\infty}(0,T;H) \cap H^1(0,T; V')} + \left\| \chi_\mu \right\|_{L^2(0,T; D(A;H)) \cap H^1(0,T; H)} \nonumber \\
+ \left\| \xi_\mu \right\|_{L^2(0,T; H)} \leq M'
\end{gather}
where $M' = M'(\Omega, T, M)$. 

Moreover, the components $(\vartheta_\mu, \chi_\mu)$ of solution continuously depend on data, in the following sense: if $(g_i, h_i, \vartheta_{0i}, \chi_{0,\mu, i})$, $i =1,2$, are two sets of data, whose norms are bounded by two constants $M_1$ and $M_2$ respectively, then the corresponding solutions $(\vartheta_{\mu,i}, \chi_{\mu,i})$ fulfil the following estimate
\begin{gather}
\int_{Q_t}{(\log \vartheta_{\mu,1} - \log \vartheta_{\mu,2}) (\vartheta_{\mu,1} - \vartheta_{\mu,2})} + \int_{Q_t}{(\xi_{\mu,1} - \xi_{\mu,2}) (\chi_{\mu,1} - \chi_{\mu,2})}  \nonumber \\
+ \int_Q{\left| \nabla (\chi_{\mu,1} - \chi_{\mu,2}) \right|^2} + \mu \int_{\Omega}{\left|\chi_{\mu,1}(t) - \chi_{\mu,2}(t) \right|^2} 
+ \left\| 1 \ast (\vartheta_{\mu,1} - \vartheta_{\mu,2}) (t) \right\|^2_V \nonumber \\
\leq M'' \left[ \left\| \eta_{0,\mu,1} - \eta_{0,\mu, 2} \right\|^2_H + \left\|g_1 - g_2\right\|^2_{L^2(0,T;H)} + \left\|h_1 - h_2\right\|^2_{L^2(0,T;L^2(\Gamma))} \right]
\end{gather}
$\forall \, t \in (0,T)$; where $\eta_{0,\mu,i} := \log \vartheta_{0i} + \chi_{0,\mu,i}$, for $i=1,2$, and $M'' = M''(\Omega, T, M_1, M_2) \in \textbf{R}_+$. 
\end{thm}

\subsection{Asymptotic behavior as $\mu \searrow 0$}
Starting from problem $(P_\mu)$, we can formally write down the limit problem $(P_0)$, consisting of the first equation (\ref{1eqPmu}) integrated in time and of the second equation (\ref{2eqPmu}) without the time derivative of the phase parameter $\chi$. 

Since we loose regularity estimates for the term $\partial _t \chi$ as $\mu \searrow 0$, the logarithmic term $\log \vartheta$ is less regular and takes values only in $V'$; thus, we have to modify the formulation of the problem and replace the term $\log \vartheta$ by a new unknown function $\zeta \in L^2(0,T; V')$ which generalizes the equality $\zeta = \log \vartheta$, through a suitable functional relation between $\zeta$ and $\vartheta$. Basically, we define an \textit{ad hoc} generalized logarithm as it has been done in  \cite{convergence}, to which we refer for a detailed discussion on this argument.

\begin{defn}
For $\vartheta \in L^2(0,T;V)$ we define $\text{Log} \, \vartheta$ as the set of $\zeta \in L^2(0,T;V')$ such that 
\begin{equation}
\langle \zeta, \theta - \vartheta \rangle + \int_Q{\psi(\vartheta)} \leq \int_Q{\psi(\theta)} \ \ \ \forall \, \theta \in L^2(0,T;V')
\end{equation}
where $\psi (\tau) = \tau (\log \tau - 1)$ if $\tau >0$, $\psi(0) =0$ and $\psi (\tau) = +\infty$ if $\tau <0$; we set $D(\text{Log}) = \left\{ \vartheta \in L^2(0,T;V) \ \left| \ \text{Log} \, \vartheta \neq \emptyset \right.  \right\}$.

At this point, we define $\Psi: L^2(0,T;V) \rightarrow (-\infty, + \infty]$ as
\begin{equation}
\Psi (v) = \int_Q{\psi(v)} \ \ \ \forall \, v \in L^2(0,T;V).
\end{equation}
$\Psi$ is proper, convex and $v \in D(\Psi)$ if and only if $v$ is non-negative. 

Moreover, it is possible to prove that $\partial \Psi: L^2(0,T;V) \rightarrow L^2(0,T;V')$ is well-defined and $\text{Log} \, v = \partial \Psi (v)$ (see \cite[Section 4]{convergence}).
\end{defn}

The problem $(P_0)$ can be written as follows
\begin{align}
&\zeta(t) + \chi(t) + 1 \ast B\vartheta (t) = 1 \ast w(t) + \eta_0 \ \ \ \text{in} \ V', \ \text{a.e. in} \ (0,T)\label{eq1P0}\\
&\zeta \in \text{Log}\,  \vartheta  \ \ \ \text{in} \ V', \ \text{a.e. in} \ (0,T) \\
&A \chi + \xi + \sigma'(\chi) = \vartheta \ \ \ \text{a.e. in} \ Q \\
&\xi \in \partial \beta(\chi) \ \ \ \text{a.e. in} \ Q, \label{ulteqP0}
\end{align}
being $\eta_0 := \log\vartheta_0 + \chi_0$; the symbol $\ast$ denotes the usual time convolution product
\begin{equation}
(a \ast b) (t) := \int_0^t{a(s) b(t-s) ds}, \ \ \ \forall \, t \in [0,T].
\end{equation}

To perform the asymptotic analysis as $\mu \searrow 0$ in problem $(P_\mu)$ (\ref{1eqPmu})-(\ref{ultimaPmu}), we ask further regularity for the functional $w$ (compare with (\ref{defw}))
\begin{equation}
\begin{split}
&g \in H^1(0,T;H)  \\
&h \in H^1(0,T; L^{2}(\Gamma)), \ h \geq 0 \ \ \ \text{a.e. in} \ \Sigma, \label{ghreg}
\end{split}
\end{equation}
which imply $w \in H^1(0,T;V')$. 

The hypotheses on initial data are the following ones
\begin{equation}
\begin{split}
&\chi_0 \in V, \ \beta(\chi_0) \in L^1(\Omega)   \\
&\vartheta_0 \in L^{\infty}(\Omega), \ \vartheta_0 > 0 \ \text{a.e. in} \ \Omega\ \ \text{e} \ 1 / \vartheta_0 \in L^{\infty}(\Omega). \label{datoinizioP0}
\end{split}
\end{equation}
We assume also that the sequence of initial data $\{\chi_{0,\mu}\}$ related to problems $(P_\mu)$ satisfies
\begin{equation}
\chi_{0,\mu} \xrightarrow{\mu \rightarrow 0} \chi_0 \ \ \ \text{in} \ V  \ \ \ \text{and} \ \ \
\left\| \chi_{0,\mu} \right\|_V + \left\| \beta(\chi_{0,\mu}) \right\|_{L^1(\Omega)} \leq c \ \ \ \forall \, \mu \in (0,1). \label{succchi0mu}
\end{equation}

Regarding the functions $\beta$ and $\sigma$, we require the same properties as before (see (\ref{betasigma})), together with some further hypotheses, which turn to be useful to prove the convergence of problem $(P_\mu)$ to $(P_0)$. In particular, we deal with two distinct hypotheses. 

\paragraph{Hypothesis 1.}
Let assume $\sigma$ to be a linear function and $\beta$ to satisfy the following condition 
\begin{equation}
\beta(r) \geq c_1 r^2 - c_2 \ \ \ \text{for each} \ r \in D(\beta),\ \text{with} \  c_1, c_2 \in \textbf{R}_+.\label{betaquadr}
\end{equation}

\paragraph{Hypothesis 2.}
Let assume the property
\begin{equation}
\Bigl[ (\xi_1 + \sigma'(\chi_1)) - (\xi_2 +\sigma'(\chi_2)) \Bigr] \left( \chi_1 - \chi_2 \right) \geq \rho \left( \chi_1 - \chi_2 \right)^2, \label{betasigmahp2}
\end{equation}
for each $\chi_i \in D(\partial \beta)$ and for each $\xi_i \in \partial \beta(\chi_i)$, $i=1,2$, for some positive constant $\rho$. 
Moreover, we ask that $\beta + \sigma$ has at least square growth
\begin{equation}
\beta(r) + \sigma(r) \geq c_1 r^2 - c_2 \ \ \ \text{for each} \ r \in D(\beta), \ \text{with} \ c_1, c_2 \in \textbf{R}_+. \label{betasigmaquadr}
\end{equation}
We can notice that the second hypothesis is fulfilled if, for example, $\partial \beta$ is strictly monotone and the growth of $\sigma'$ is dominated by $\partial \beta$.

\begin{thm}
\label{teo3}
Let $\mu \in (0,1)$. Assuming the same hypotheses of Theorem \ref{teo1}, let suppose that (\ref{ghreg}), (\ref{datoinizioP0}), (\ref{succchi0mu}) hold and let assume either Hypothesis 1 or 2 on the functions $\beta$ and $\sigma$.  

Then, the solution to problem $(P_\mu)$ (\ref{1eqPmu})-(\ref{ultimaPmu}), given by Theorem \ref{teo1}, converges to a solution to problem $(P_0)$ (\ref{eq1P0})-(\ref{ulteqP0}), as $\mu \searrow 0$, with respect to the natural topologies and with the following regularities
\begin{equation}
\begin{split}
&\vartheta \in L^2(0,T;V) \\
&\zeta \in L^{\infty}(0,T;V') \\
&\chi \in L^2(0,T;D(A;H))  \\
&\xi \in L^2(0,T;H).
\end{split}
\end{equation}
\end{thm}

\begin{thm}
\label{teo4}
Under the same assumptions of Theorem \ref{teo3}, the components $(\vartheta, \chi)$ of the solution continuously depend on data; i.e. if $(g_i, h_i, \vartheta_{0i}, \chi_{0i})$, $i =1,2$, are two sets of data, whose norms are bounded by constants $M_1$ and $M_2$ respectively, then the corresponding solutions $(\vartheta_i, \chi_i)$ fulfil the following estimates. Under Hypothesis 1, we have
\begin{gather}
\int_{Q_t}{(\zeta_1 - \zeta_2) (\vartheta_1 - \vartheta_2)} + \int_{Q_t}{(\xi_1 - \xi_2) (\chi_1 - \chi_2)} 
+ \left\| \nabla (\chi_1 -\chi_2) \right\|^2_{L^2(Q)} +\nonumber \\
\left\| 1 \ast (\vartheta_1 - \vartheta_2) (t) \right\|^2_V  
\leq \tilde{M} \left[ \left\| \eta_{0,1} - \eta_{0,2} \right\|^2_H + \left\|g_1 - g_2\right\|^2_{L^2(0,T;H)} + \left\|h_1 - h_2\right\|^2_{L^2(0,T;L^2(\Gamma))} \right]
\end{gather}
$\forall \, t \in (0,T)$; where $\eta_{0i} = \log \vartheta_{0i} + \chi_{0i}$, $i=1,2$, and $\tilde{M} = \tilde{M}(\Omega, T, M_1, M_2) \in \textbf{R}_+$; while, under Hypothesis 2, we have
\begin{gather}
\int_{Q_t}{(\zeta_1 - \zeta_2) (\vartheta_1 - \vartheta_2)} 
+ \left\| \chi_1 - \chi_2 \right\|_{L^2(0,T;V)} + \left\| 1 \ast (\vartheta_1 - \vartheta_2) (t) \right\|^2_V \nonumber \\
\leq \overline{M} \left[ \left\| \eta_{0,1} - \eta_{0,2} \right\|^2_H + \left\|g_1 - g_2\right\|^2_{L^2(0,T;H)} + \left\|h_1 - h_2\right\|^2_{L^2(0,T;L^2(\Gamma))} \right],
\end{gather}
$\forall \, t \in (0,T)$, with $\overline{M} = \overline{M}(\Omega, T, M_1, M_2) \in \textbf{R}_+$. 
\end{thm}

The following Sections \ref{exist} and \ref{mu} are devoted to the proof of the previous results.

\section{Proof of Theorem \ref{teo1}}
\label{exist}

In order to prove the well-posedness of $(P_\mu)$, $\mu >0$, we proceed by a double approximation of the problem: first we introduce a family of regularized problems $\{(P_\varepsilon) \}$ depending on a positive parameter $\varepsilon$; next, we prove the existence of a solution to the problem $(P_\varepsilon)$, $\varepsilon >0$ fixed, applying a standard Faedo-Galerkin method. Afterwards, we let $\varepsilon$ tend to zero and deal with the original problem $(P_\mu)$.

We base our proof on the guidelines of article \cite[Section 3, 4 and 5]{global}. The main differences with respect to \cite{global} are due to our boundary conditions: we will carefully detail this point in the proof.

\subsection{Hypotheses and preliminary results}
We consider the Yosida regularizations $\beta'_\varepsilon$ and $\log_\varepsilon$ of the maximal monotone graph $\partial \beta$ and $\log$ respectively (see Reference \cite[Chap. II]{brezis}), and define $\beta_\varepsilon$, $\text{Log}_\varepsilon: \textbf{R} \rightarrow \textbf{R}$ by
\begin{equation}
\beta_\varepsilon(r) := \int_0^r{\beta'_{\varepsilon}(s)ds} \ \ \ \text{and} \ \ \ \text{Log}_\varepsilon (r) := \varepsilon r + \log_\varepsilon(r).
\end{equation}
We can notice that both $\beta'_\varepsilon$ and $\text{Log}_\varepsilon$ are monotone and Lipschitz continuous. 

We need one more function, namely
\begin{equation}
I_{\varepsilon} (r) := \int_0^r{s \text{Log}'_{\varepsilon}(s) ds}, \ \ r\in \textbf{R},
\end{equation}
which is an approximation of the identity on $(0,+ \infty)$.

At this point, we state two lemmas which will be useful later.
\begin{lemma}[see Lemma 4.2, \cite{global}]
The function $\text{Log}_\varepsilon$ is differentiable, with derivative $\text{Log}'_\varepsilon$ such that
\begin{equation}
\varepsilon \leq \text{Log}'_\varepsilon (r) \leq \varepsilon + \frac{1}{r}, \ \ \ \forall \, r \in \textbf{R},
\end{equation}
provided that $\varepsilon$ is small enough.
\label{logprimo}
\end{lemma} 

\begin{lemma}[see Lemma 4.3, \cite{global}]
For all $r \in \textbf{R}$ the following inequality holds
\begin{equation}
I_{\varepsilon}(r) \leq \frac{\varepsilon}{2} r^2 + 2r,
\end{equation}
provided that $\varepsilon$ is sufficiently small.
\label{iepsilon}
\end{lemma}

Finally, we introduce regularized data $\{ \vartheta_{0 \varepsilon}\}$ such that
\begin{equation}
\begin{split}
&\vartheta_{0 \varepsilon} \in V \ \ \forall \, \varepsilon >0, \ \ \ 0< \vartheta_* \leq \vartheta_{0 \varepsilon} \leq \vartheta^* \ \ \text{a.e. in} \ \Omega, \ \forall \, \varepsilon > 0,   \\ 
&\vartheta_{0 \varepsilon} \xrightarrow{\varepsilon \rightarrow 0}  \vartheta_0 \ \ \ \text{in} \ H \ \text{and a.e.} \ \Omega. \label{theta0eps}
\end{split}
\end{equation}

Hence, the approximating problem $(P_\varepsilon)$ is the following one
\begin{align}
&\partial _t \left( \text{Log}_{\varepsilon} \vartheta_{\varepsilon} + \chi_{\varepsilon} \right) + B\vartheta_{\varepsilon} = w  \ \ \ \text{in} \ V', \ \text{a.e. in} \ (0,T) \label{probleps1}\\
&\mu \partial_t \chi_{\varepsilon} + A \chi_{\varepsilon} + \beta'_\varepsilon(\chi_{\varepsilon}) + \sigma'(\chi_\varepsilon) = \vartheta_{\varepsilon} \ \ \ \text{a.e. in} \ Q \label{probleps2} \\
&\vartheta_{\varepsilon} (0) = \vartheta_{0\varepsilon} \ \ \ \text{and} \ \ \ \chi_{\varepsilon}(0) = \chi_{0,\mu} \ \ \ \text{a.e. in} \ \Omega. \label{problepsult}
\end{align}

\begin{thm}
\label{teoeps}
Let $\varepsilon > 0$ fixed. Under the same hypotheses of Theorem \ref{teo1} and with the further hypotheses (\ref{theta0eps}), there exists a unique pair $(\vartheta_\varepsilon, \chi_\varepsilon)$ such that
\begin{equation}
\begin{split}
&\vartheta_\varepsilon \in L^2(0,T;V) \cap H^1(0,T;H)  \\
&\chi_\varepsilon \in L^2(0,T;D(A;H)) \cap H^1(0,T;H) 
\end{split}
\end{equation}
and fulfilling the equations (\ref{probleps1})--(\ref{problepsult}) of problem $(P_\varepsilon)$.
\end{thm}

The next subsection is devoted to the proof of Theorem \ref{teoeps}.

\subsubsection{Proof of Theorem of existence of a solution to the problem $(P_\varepsilon)$}

\paragraph{Discrete problem.}
We proceed by using the standard Faedo-Galerkin procedure. First of all, let's introduce two increasing sequences $V_n$ and $W_n$ of finite dimensional subspaces of $V$, such that 
\begin{equation}
\overline{\bigcup_{n=0}^{\infty} V_n} =V \ \ \ \text{and} \ \ \  \overline{\bigcup_{n=0}^{\infty} W_n} =V;
\end{equation}
in particular, we choose these subspaces in such a way that $V_n \subseteq D(B;H)$ and $W_n \subseteq D(A;H)$. 

Moreover, we approximate the initial data by appropriate sequences of data such that
\begin{equation}
\begin{split}
\vartheta_{0,n} \in V_n \ \ \forall \, n, \ \ \ & \vartheta_{0,n} \xrightarrow{n \rightarrow +\infty} \vartheta_{0 \varepsilon} \ \ \ \text{in} \  V   \\ 
\chi_{0,n} \in W_{n} \ \ \forall \, n, \ \ \ & \chi_{0,n} \xrightarrow{n \rightarrow +\infty} \chi_{0} \ \ \ \text{in} \ V. \label{chi0ntheta0n}
\end{split}
\end{equation}

The discrete problem $(P_n)$ is the following one
\begin{align}
&\left( \partial_t(\text{Log}_\varepsilon \vartheta_n(t) + \chi_n(t)), v \right) _H + (B \vartheta_n(t), v)_H = (w(t), v)_H  \nonumber \\
&\text{for a.e.} \ t \in (0,T), \ \forall\,  v \in V_n   \label{eq1n} \\
&\mu (\partial_t \chi_n(t), u)_H + (\nabla \chi_n(t), \nabla u)_H + (\beta'_\varepsilon(\chi_n (t)) + \sigma'(\chi_n(t)),u)_H = (\vartheta_n (t), u)_H  \nonumber\\
&\text{for a.e.} \ t \in (0,T), \ \forall \, u \in W_{n} \label{eq2n}  \\
&\vartheta_n(0)=\vartheta_{0,n} \ \ \ \text{and} \ \ \ \chi_n (0) = \chi_{0,n} \ \ \ \ \text{a.e. in} \ \Omega. \label{ultimaeqn}
\end{align}

\begin{thm}
Let $n \in \textbf{N}$. Under the hypotheses of Theorem \ref{teoeps} and assuming $\vartheta_{0,n} \in V_n$ and $\chi_{0,n} \in W_{n}$,
the discrete problem $(P_n)$ (\ref{eq1n})-(\ref{ultimaeqn}) admits a unique pair $(\vartheta_n, \chi_n)$ as solution, such that
\begin{equation}
\vartheta_n \in C^1([0,T]; V_n) \ \ \ \text{and} \ \ \ \chi_n \in C^1([0,T];W_{n}).
\end{equation}
\end{thm}

In order to prove this theorem, let $\{ e_j \}_{j=1}^{n}$ and $\{ b_{j} \}_{j=1}^{n}$ be two bases for $V_n$ and $W_{n}$ respectively. Since we can express the functions $\vartheta_n$ and $\chi_n$ as linear combinations of these bases, the true unknowns are the coefficients $u_j$ and $y_j$ of such representations. 
If $\textbf{u}$ and $\textbf{y}$ are the vectors which collect these coefficients, then the system (\ref{eq1n})-(\ref{eq2n}) may be rewritten in the form of a system of integro-differential ordinary equations
\begin{equation} 
\textbf{E}\left(t, \textbf{u}(t), \textbf{y}(t), \textbf{u}'(t), \textbf{y}'(t)  \right) = 0 \ \ \ t \in (0,T),\label{Eimplicit}
\end{equation}
where $\textbf{E} = (\textbf{F}, \textbf{G})$, with components $F_j$ and $G_j$ of $\textbf{F}$ and $\textbf{G}$ defined as follows
\begin{align}
F_j &= \int_{\Omega}{\text{Log}'_\varepsilon \left( \sum_i{u_i e_i}\right) \left( \sum_i{u'_i e_i}\right) e_j} 
+ \int_{\Omega}{\left( \sum_i{y'_i b_{i}} \right)e_j}  
+ \sum_i u_i \int_{\Omega}{\nabla e_i \cdot \nabla e_j} \nonumber \\
& + \int_{\Gamma}{\alpha \left( \sum_i{u_i e_i} \right) e_j} - \int_{\Omega}{w e_j} \label{compF}\\
G_j &= \mu \int_{\Omega}{\left( \sum_i{y'_i b_{i}} \right)b_{j}} 
+ \int_{\Omega}{\left( \sum_i{y_i \nabla b_{i}} \right)\cdot \nabla b_{j}}
+ \int_{\Omega}{\beta'_\varepsilon \left( \sum_i{y_i b_{i}} \right) b_{j}} \nonumber \\
&+ \int_{\Omega}{\sigma' \left( \sum_i{y_i b_{i}} \right) b_j} - \int_{\Omega}{\left( \sum_i{u_i e_i} \right) b_{j}} \label{compG} 
\end{align}
for $j=1, \cdots, n$; we can notice that the variables $(\textbf{u}', \textbf{y}')$ are independent, as well as $(\textbf{u}, \textbf{y})$. 

The first target is to put the system in a normal form. After that, we will apply the Implicit Function Theorem to prove the existence of a solution (at least local). 
We can notice that $\textbf{E}$ is a continuous function and has continuous derivatives with respect to $\textbf{u}'$ and $\textbf{y}'$. Next, we study its Jacobian matrix with respect to the above variables: we can think of it as a four-block matrix, namely the derivatives of $\textbf{F}$ and $\textbf{G}$ with respect to $\textbf{u}'$ and $\textbf{y}'$. Since $\textbf{G}$ does not depend on $\textbf{u}'$, the determinant of the Jacobian is
\begin{equation} 
\det \left[ \frac{\partial \textbf{E}}{\partial (\textbf{u}', \textbf{y}')} \right] = \det \left[ \frac{\partial \textbf{F}}{\partial \textbf{u}'} \right] \cdot \det \left[ \frac{\partial \textbf{G}}{\partial \textbf{y}'} \right].
\end{equation}
Moreover, all the partial derivatives involved are scalar products: indeed,
\begin{equation} 
\frac{\partial G_j}{\partial y'_i} = \mu \int_{\Omega}{b_{i} b_{j}} = \mu (b_{i}, b_{j})_H; \ \ \ 
\frac{\partial F_j}{\partial u'_i} = \int_{\Omega}{\text{Log}'_\varepsilon \left( \sum_k{u_k(t) e_k} \right) e_i e_j} =: (e_i, e_j)_{t,\textbf{u}};
\end{equation}
where $( \cdot , \cdot)_{t,\textbf{u}}$ is an equivalent scalar product in $H$, thanks to Lemma \ref{logprimo}. Hence, the above matrices are positive definite. 

Now, we find a point $(t_*, \textbf{u}_*, \textbf{y}_*, \textbf{u}'_*, \textbf{y}'_*)$ such that $E(t_*, \textbf{u}_*, \textbf{y}_*, \textbf{u}'_*, \textbf{y}'_*) = 0$. First of all, we set $t_* = 0$; then, we choose $\textbf{u}_*$ and $\textbf{y}_*$ equal to the vectors of coefficients of the initial data $\vartheta_{0,n} \in V_n$ $\chi_{0,n} \in W_n$ with respect to the chosen bases. 

We define $\chi^\bullet_{0,n}$ as the solution of the following equation
\begin{equation}
\mu \chi_{0,n}^{\bullet} - \Delta \chi_{0,n} + \beta'_\varepsilon (\chi_{0,n}) + \sigma'(\chi_{0,n}) = \vartheta_{0,n}.
\end{equation}
By comparison, we have $\chi^\bullet_{0,n} \in H$; then, we take its projection $\chi_{0,n}^{\bullet*}$ on $W_n$ with respect to the scalar product in $H$ and we set $\textbf{y}'_*$ equal to the coefficients of $\chi_{0,n}^{\bullet*}$ with respect to the basis $\{b_j \}$. 

Next, in order to find $\textbf{u}'_*$, we define 
\begin{equation}
w_n (t) = \sum_{i=1}^{n} w_i(t) e_i,
\end{equation} 
with $w_i (t) = (w(t),e_i)_H$; by construction, $w_n \in C^0([0,T];V_n)$ and $w_n \rightarrow w$ in $C^0([0,T];V')$, as $n \nearrow + \infty$. Then, we define $u_{0,n}^{\bullet}$ as the solution of the following equation
\begin{equation} 
\text{Log}'_\varepsilon(\vartheta_{0,n})u_{0,n}^{\bullet} + \chi_{0,n}^{\bullet *} + B\vartheta_{0,n} = w_n(0). 
\end{equation}
By comparison, we have $u_{0,n}^{\bullet} \in H$; then, as before, we take its projection on $V_n$ with respect to the scalar product $(\cdot, \cdot)_{t_*,\textbf{u}_*}$ and we set $\textbf{u}'_*$ equal to the coefficients of the projection with respect to the base $\{ e_j \}$. 

Therefore, we have $\textbf{E}\left(t_*, \textbf{u}_*, \textbf{y}_*, \textbf{u}'_*, \textbf{y}'_*  \right) = 0$ and we can apply the Implicit Function Theorem to conclude that the system is (locally) equivalent to a system of the form
\begin{equation} 
\left( \textbf{u}'(t), \textbf{y}'(t) \right) = \mathscr{E}(t, \textbf{u}(t), \textbf{y}(t)), \label{ptofix}
\end{equation}
where $\mathscr{E}$ is a Lipschitz continuous function with respect to the variables $\textbf{u}$ and $\textbf{y}$. 

At this point, we can integrate the expression above and we get
\begin{equation} 
\left( \textbf{u}(t), \textbf{y}(t) \right) = ( \textbf{u}_{0,n}, \textbf{y}_{0,n}) + \int_0^{t}{\mathscr{E}(s, \textbf{u}(s), \textbf{y}(s)) ds} =: F(t, \textbf{u}(t), \textbf{y}(t)) 
\end{equation}
where $F(t, \textbf{u}(t), \textbf{y}(t))$ turns out to be a contraction operator in the space $C^0([0,\tau]; \textbf{R}^{2n})$ for $0< \tau < T$ small enough. Using the Contraction Theorem, we claim that there exists a unique local solution of the discrete problem. 

Moreover, the solution is indeed a global solution defined on the whole interval $[0,T]$. This result can be proved noting that solution $(\vartheta_n, \chi_n)$ are continuous with respect to the time variable; thus, we can consider a new Cauchy problem with initial data $(\vartheta_n (\tau), \chi_n (\tau)) \in V_n \times W_n$ and apply the same method as above to find a new pair of solutions with the same regularity of the previous pair which extends it to the interval $(\tau, 2\tau)$. 

In general, repeating iteratively this procedure, we can find a global solution 
\begin{equation}
\vartheta_n \in C^1([0,T]; V_n) \ \ \ \text{and} \ \ \ \chi_n \in C^1([0,T]; W_{n}).
\end{equation}

\begin{remark}
We could have weakened the hypotheses on the source data and considered an inner heat supply with the following regularity $g \in L^2(0,T;H)$, while the regularity of the $h$-term remains the same.

In this case, in order to prove the theorem regarding the discrete problem, we should have approximated the source operator $w \in L^2(0,T;V')$ with a sequence of functions $\{w_k \} \subseteq C^0([0,T];V')$. 

Furthermore, we identify a new sequence of functions $\{w_n \} \subseteq C^0([0,T];V_n)$ defined as
\begin{equation}
w_n(t) = \sum_{i=1}^{n} (w_k(t), e_i)_H \, e_i,
\end{equation}
so that we can find the point $(t_*, \textbf{u}_*, \textbf{y}_*, \textbf{u}'_*, \textbf{y}'_*)$ such that $\textbf{E}(t_*, \textbf{u}_*, \textbf{y}_*, \textbf{u}'_*, \textbf{y}'_*) = 0$. 

For the sake of simplicity, we preferred to introduce a stronger assumption on $w$. 
\end{remark}

\paragraph{Uniform estimates with respect to $n$.}
In order to prove the existence of a solution to the problem $(P_\varepsilon)$, $\varepsilon >0$ fixed, we perform some a priori estimates and then we let $n$ tend to infinity. 

In this section and in the following ones, we assume the arrangement that the symbol $c$ denotes a constant which depends only on the given data; the exact value of the constant $c$ may vary in the different estimates and even in the same chain of inequalities. A notation like $c_\varepsilon$ stands for a constant which depends not only on the data, but also on the parameter $\varepsilon$; nevertheless, it does not depend on the index $n$. We will also use constants like $c_\delta$, which denotes a dependence on a positive parameter $\delta$, but still independent of $n$.

The basic estimate is obtained as follows: we test the first equation (\ref{eq1n}) by $\vartheta_n(t)$ and the second equation (\ref{eq2n}) by $\partial_t \chi_n(t)$, we integrate over $(0,t)$, $t \in (0,T]$, and we add the equations to each other
\begin{gather}
\int_{Q_t}{\partial_t(\text{Log}_\varepsilon \vartheta_n) \vartheta_n} + 
\int_{Q_t}{\left| \nabla \vartheta_n \right|^2} + 
\int_{\Sigma_t}{\alpha \left|\vartheta_n\right|^2} +
\mu \int_{Q_t}{\left| \partial_t\chi_n \right|^2} + 
\frac{1}{2} \int_{\Omega}{\left| \nabla \chi_n (t) \right|^2}  \nonumber \\
+ \int_{\Omega}{\beta_\varepsilon(\chi_n(t))} 
= \int_{Q_t}{w \vartheta_n} 
- \int_{Q_t}{\sigma'(\chi_n)\partial_t \chi_n} 
+ \int_{\Omega}{\beta_\varepsilon(\chi_{0,n})}
+ \frac{1}{2} \int_{\Omega}{\left| \nabla \chi_{0,n} \right|^2}.
\label{eqstima1n}
\end{gather}

For the sake of clearness, we deal with each term separately. First of all, we notice that there are some boundary terms (due to the Robin condition that we imposed for $\vartheta$) that have to be properly estimated: to do so, we will make use of the Trace Theorem. 

Let's start with the first term, which can be rewritten as
\begin{equation}
\int_{Q_t}{\partial_t(\text{Log}_\varepsilon \vartheta_n) \vartheta_n} = \int_{\Omega}{I_\varepsilon ( \vartheta_n (t))} - \int_{\Omega}{I_\varepsilon ( \vartheta_{0,n})};
\end{equation}
moreover, the last term is moved on the right hand side of the equality (\ref{eqstima1n}) and it is uniformly bounded, thanks to Lemma \ref{iepsilon} and (\ref{chi0ntheta0n}).

Thanks to (\ref{alpha}), the second and third terms are estimated from below as follows
\begin{equation}
\int_{Q_t}{\left| \nabla \vartheta_n \right|^2} + \int_{\Sigma_t}{\alpha \left|\vartheta_n\right|^2} \geq c \left\| \vartheta_n  \right\|^2_{L^2(0,t; V)}.
\end{equation}  

The source term on the right hand side of (\ref{eqstima1n}) can be easily handled, using H$\ddot{\text{o}}$lder and Young Inequalities together with the Trace Theorem: 
\begin{equation}
 \int_{Q_t}{g \vartheta_n} + \int_{\Sigma_t}{h \vartheta_n} \leq \delta  \left\| \vartheta_n \right\|^2_{L^2(0,t; V)} + c_\delta \left[\left\| g\right\|^2_{L^2(Q_t)} +  \left\| h\right\|^2_{L^2(\Sigma_t)} \right].
\end{equation}
Moreover, we estimate the $\beta_\varepsilon$-term in the following way
\begin{gather}
\int_{\Omega}{ \beta_\varepsilon(\chi_{0,n})} + \frac{1}{2} \int_{\Omega}{\left| \nabla \chi_{0,n} \right|^2}
= \int_{\Omega}{\left[ \int_0^{\chi_{0,n}}{\beta'_\varepsilon (s) ds} \right]} + \frac{1}{2} \int_{\Omega}{\left| \nabla \chi_{0,n} \right|^2} \nonumber \\
\leq \frac{c}{\varepsilon} \int_{\Omega}{\left|\chi_{0,n} \right|^2} + \frac{1}{2} \int_{\Omega}{\left| \nabla \chi_{0,n} \right|^2} \leq c_\varepsilon,
\end{gather}
thanks to the Lipschitz property of $\beta'_\varepsilon$ and (\ref{chi0ntheta0n}). 

To estimate the $\sigma'$-term, we have to add on both sides of the equation (\ref{eqstima1n}) the term $1/2 \left\| \chi_n(t) \right\|^2_H$: on the left hand side a $V$-norm of $\chi_n(t)$ appears, while on the right hand side we perform a chain of inequalities, using (\ref{stimasigma}), (\ref{chi0ntheta0n}) and Young Inequality, 
\begin{gather}
\frac{1}{2} \left\| \chi_n(t) \right\|^2_H - \int_{Q_t}{\sigma'(\chi_n)\partial_t \chi_n} = \int_{Q_t}{\chi_n \left( \partial_t \chi_n \right)} + \frac{1}{2} \left\| \chi_{0,n} \right\|^2_H - \int_{Q_t}{\sigma'(\chi_n)\partial_t \chi_n} \nonumber \\
\leq \frac{1}{2} \left\| \chi_{0,n} \right\|^2_H + \int_0^t{\int_{\Omega}{\chi_n \left( \partial_t \chi_n \right)}} + c_\sigma \int_0^t{\int_{\Omega}{\left( 1 + \left| \chi_n \right| \right) \partial_t \chi_n}} \nonumber \\
\leq c \left\| \chi_{0,n} \right\|^2_H + \delta \left\| \partial_t \chi_n \right\|^2_{L^2(Q_t)} + c_\delta \int_0^t{\left\| \chi_n(s) \right\|^2_H ds} \nonumber \\
\leq c \left\| \chi_{0,n} \right\|^2_H + \delta \left\| \partial_t \chi_n \right\|^2_{L^2(Q_t)} + c_\delta \int_0^t{\left\| \chi_n(s) \right\|^2_V ds}, 
\end{gather}
for each $\delta >0$.

Finally, collecting all the previous results and choosing $\delta >0$ small enough, we can apply the Gronwall Lemma (see Reference \cite[Appendix, Lemma A.4]{brezis}) and get
\begin{gather}
\left\| I_\varepsilon(\vartheta_n) \right\|_{L^{\infty}(0,T; L^1(\Omega))} + \left\| \vartheta_n  \right\|^2_{L^2(0,T; V)} + \left\| \chi _n \right\|^2_{H^1(0,T;H)} + \left\| \chi _n \right\|^2_{L^{\infty}(0,T;V)} \nonumber \\ 
+ \left\| \beta_\varepsilon (\chi_n) \right\|_{L^{\infty}(0,T; L^1(\Omega))} \leq c_\varepsilon.
\label{stima1n}
\end{gather}
after having taken the supremum over $t \in (0,T)$. 

Next, we derive an \textit{a priori} bound for $\partial_t \vartheta_n \in L^2(0,T;H)$; we test (\ref{eq1n}) by $\partial_t \vartheta_n (t)$ and integrate over $(0,t)$: 
\begin{gather}
\int_{Q_t}{\text{Log}'_\varepsilon (\vartheta_n) \left| \partial_t \vartheta_n \right|^2} + \int_{Q_t}{(\partial_t \chi_n)(\partial_t \vartheta_n)} + \frac{1}{2} \left\|\nabla \vartheta_n (t) \right\|^2_H + \int_{\Sigma_t}{\alpha \vartheta_n (\partial_t \vartheta_n)}  \nonumber \\ 
=\int_{Q_t}{g_n \partial_t \vartheta_n} + \int_{\Sigma_t}{h_n \partial_t \vartheta_n} + \frac{1}{2} \left\|\nabla \vartheta_{0,n} \right\|^2_H.
\label{eqstima2n}
\end{gather}

Thanks to Lemma \ref{logprimo}, the first term is easily estimated
\begin{equation}
\int_{Q_t}{\text{Log}'_\varepsilon (\vartheta_n) \left| \partial_t \vartheta_n \right|^2} \geq \varepsilon \left\| \partial_t \vartheta_n \right\|^2_{L^2(0,t;H)}.
\end{equation}
Concerning the second term, we move it the right hand side and we get
\begin{equation}
\left| \int_{Q_t}{\partial_t \chi_n \, \partial_t \vartheta_n} \right|  \leq \delta \left\| \partial_t \vartheta_n \right\|^2_{L^2(0,t;H)} + c_\delta \left\| \partial_t \chi_n \right\|^2_{L^2(0,T;H)},
\end{equation}
where we applied H$\ddot{\text{o}}$lder and Young Inequalities; moreover, the norm of $\partial_t \chi_n $ is uniformly bounded, thanks to the previous estimate (\ref{stima1n}).

Due to (\ref{alpha}), (\ref{chi0ntheta0n}) and the Trace Theorem, the boundary term is treated as follows
\begin{equation}
\int_{\Sigma_t}{\alpha \vartheta_n (\partial_t \vartheta_n) } \geq \frac{\overline{\alpha}}{2} \left\| \vartheta_n (t) \right\|^2_{L^2(\Gamma)} - c \left\| \vartheta_{0,n} \right\|^2_{V} \geq \frac{\overline{\alpha}}{2} \left\| \vartheta_n (t) \right\|^2_{L^2(\Gamma)} - c.
\end{equation}

Finally, after having noticed that $\left\|\nabla \vartheta_{0,n} \right\|^2_H$ is uniformly bounded (see (\ref{chi0ntheta0n})), we estimate the source terms:
\begin{eqnarray}
\int_{Q_t}{g \partial_t \vartheta_n} &\leq& \delta \left\| \partial_t \vartheta_n \right\|^2_{L^2(0,t;H)} + c_\delta \left\| g \right\|^2_{L^2(0,T;H)};  \nonumber \\
\int_{\Sigma_t}{h \partial_t \vartheta_n} &=& \int_{\Gamma}{h(t) \vartheta_n(t)} - \int_{\Gamma}{h(0) \vartheta_{0,n}} - \int_{\Sigma_t}{(\partial_t h) \vartheta_n}  \nonumber \\
&\leq& \delta \left\| \vartheta_n(t)\right\|^2_{L^2(\Gamma)} + c \left\| h \right\|^2_{L^{\infty}(0,T;L^2(\Gamma))} + \left\| \vartheta_{0,n} \right\|^2_{V}  \nonumber \\
&& + \int_0^t{\left\| \partial_t h (s) \right\|_{L^2(\Gamma)} \left\| \vartheta_n (s) \right\|_{L^2(\Gamma)} ds};
\end{eqnarray}
we recall that $h \in L^2(0,T;L^2(\Gamma)) \cap W^{1,1}(0,T;L^2(\Gamma))$.

Collecting all the estimates above, we can choose $\delta >0$ small enough and apply the Gronwall Lemma (see Reference \cite[Appendix, Lemma A.5]{brezis}), in order to get
\begin{equation}
\left\| \vartheta_n \right\|^2_{L^{\infty}(0,T;V)} + \left\| \partial_t \vartheta_n \right\|^2_{L^2(0,T;H)} \leq c_\varepsilon.
\end{equation}

\paragraph{Passage to the limit as $n \nearrow + \infty$.}
Using the previous estimates, we can claim that there exist 
\begin{equation}
\vartheta_\varepsilon \in H^1(0,T; H) \cap L^{\infty}(0,T;V) \ \ \ \text{and} \ \ \ \chi_\varepsilon \in H^1(0,T; H) \cap L^{\infty}(0,T;V)
\end{equation}
such that $\vartheta_n \rightharpoonup \vartheta_\varepsilon$ and $ \chi_n \rightharpoonup \chi_\varepsilon$ weakly in $H^1(0,T; H) \cap L^{\infty}(0,T;V)$, as $n \nearrow + \infty$ (at least for a subsequence).

These weak convergences are already sufficient to ensure that the Cauchy conditions hold in the limit: $\vartheta_\varepsilon (0) = \vartheta_{0,\varepsilon}$ and $\chi_\varepsilon (0) = \chi_{0,\mu}$ in $H$ and a.e. in $\Omega$. Moreover, applying Aubin Lemma (see Reference \cite[Appendix, Lemma A.5]{brezis}), we have 
\begin{equation}
\vartheta_n \rightarrow \vartheta_\varepsilon \ \ \ \text{and} \ \ \ \chi_n \rightarrow \chi_\varepsilon \ \ \ \text{in} \ L^2(0,T;H).
\end{equation}
Strong convergence implies pointwise convergence (a.e.) in $Q$; therefore, we can identify the limits of the non-linear terms, taking into account the Lipschitz property of $\beta'_\varepsilon$, $\text{Log}_\varepsilon$ and $\sigma'$. 

At this point, we can show that the pair $(\vartheta_\varepsilon, \chi_\varepsilon)$ is indeed a solution to problem $(P_\varepsilon)$, $\varepsilon >0$. Fix $n \in \textbf{N}$ and test equations (\ref{eq1n}) and (\ref{eq2n}) by two arbitrary functions $v \in L^2(0,T;V_m)$ and $u \in L^2(0,T;W_{m})$ respectively, with $m\leq n$; then, we integrate over $(0,T)$. Now, we let $n$ tend to infinity and, using the proved convergences, we get
\begin{equation}
\begin{split}
&\int_{Q}{\partial_t(\text{Log}_\varepsilon \vartheta_\varepsilon + \chi_\varepsilon) v} + \int_{Q}{\nabla \vartheta_\varepsilon \cdot \nabla v} + \int_{\Sigma}{\alpha \vartheta_\varepsilon v} = \int_{Q}{g v} + \int_{\Sigma}{h v}  \\
&\mu \int_{Q_t}{ (\partial_t\chi_\varepsilon) u} + \int_Q{ \nabla \chi_\varepsilon \cdot \nabla u} + \int_Q{\beta'_\varepsilon(\chi_\varepsilon) u} + \int_{Q}{\sigma'(\chi_\varepsilon) u} = \int_{Q}{\vartheta_\varepsilon u} \label{2eqvareps}
\end{split}
\end{equation}
$\forall \, v \in L^2(0,T;V_m)$ and $\forall \, u \in L^2(0,T;W_{m})$, $\forall \, m \in \textbf{N}$. Since $m$ is arbitrary, the same variational equations hold for each $v \in L^2(0,T;V)$, by density.

Moreover, by comparison in the second equation of (\ref{2eqvareps}), we have $- \Delta \chi_\varepsilon \in L^2(0,T;H)$, then, due to Elliptic Regularity Theorem (see \cite[Chap. 1.5]{traccia}), we can conclude
\begin{equation}
\chi_\varepsilon \in L^2(0,T;D(A;H)).
\end{equation}

\paragraph{Uniqueness of solution to the problem $(P_\varepsilon)$.}
In order to prove the continuous dependence of solutions from data (hence, uniqueness), we follow the same method illustrated in \cite[Sec. 5]{global}.

Let consider the time-integrated version of (\ref{probleps1}), namely
\begin{equation}
\text{Log}_{\varepsilon} \vartheta_{\varepsilon} + \chi_{\varepsilon} + 1 \ast B\vartheta_{\varepsilon} = 1 \ast w + \eta_{0, \varepsilon} \ \ \ \text{where} \ \eta_{0,\varepsilon} := \text{Log}_\varepsilon \vartheta_{0,\varepsilon} + \chi_{0,\mu}, \label{probleps1int}
\end{equation}
and we couple it with the second equation (\ref{probleps2}) of $(P_\varepsilon)$. We pick two solutions $(\vartheta_{\varepsilon,i}, \chi_{\varepsilon,i})$ to the system corresponding to the sets of data $(w_i, \eta_{0,\varepsilon,i})$, $i=1,2$. Then we write both (\ref{probleps1int}) and (\ref{probleps2}) for such solutions and multiply the difference of the first equations by $\vartheta_\varepsilon := \vartheta_{\varepsilon,1} - \vartheta_{\varepsilon,2}$ and the difference of the second ones by $\chi_\varepsilon := \chi_{\varepsilon,1} - \chi_{\varepsilon,2}$. Finally, we sum the obtained equalities to each other and integrate over $Q_t := \Omega \times (0,t)$.

After some manipulations, we have 
\begin{gather}
\int_{Q_t}{[\text{Log}_\varepsilon \vartheta_{\varepsilon,1} - \text{Log}_\varepsilon \vartheta_{\varepsilon,2}] \vartheta_\varepsilon} 
+ \frac{1}{2} \left\| 1 \ast \nabla \vartheta_\varepsilon (t) \right\|^2_H
+ \int_{\Sigma_t}{(1 \ast \alpha \vartheta_\varepsilon) \vartheta_\varepsilon} 
+ \frac{\mu}{2} \left\| \chi_\varepsilon(t) \right\|^2_H \nonumber \\ 
+ \int_{Q_t}{\left|\nabla \chi_\varepsilon \right|^2} 
+ \int_{Q_t}{\left[ \beta'_\varepsilon(\chi_{\varepsilon,1}) - \beta'_\varepsilon(\chi_{\varepsilon,2}) \right] \chi_\varepsilon} 
= \int_{Q_t}{(1 \ast w + \eta_{0,\varepsilon}) \vartheta_\varepsilon} \nonumber \\
- \int_{Q_t}{[\sigma'(\chi_{\varepsilon,1}) - \sigma'(\chi_{\varepsilon,2})] \chi_\varepsilon}
+\frac{\mu}{2} \left\| \chi_{0,\mu} \right\|^2_H, \label{unicitaeps}
\end{gather}
where we introduced a similar notation for all the differences involved: $w := w_1 - w_2$, $\chi_{0,\mu}:= \chi_{0,\mu 1} - \chi_{0,\mu 2}$.

We deal with each term separately. Concerning the non linear terms, we have that the $\text{Log}_\varepsilon$-term and the $\beta'_\varepsilon$-term are non-negative, since $\text{Log}_\varepsilon$ and $\beta'_\varepsilon$ are monotone, while the $\sigma'$-term is estimated as follow (see (\ref{betasigma}))
\begin{equation}
\int_{Q_t}{[\sigma'(\chi_{\varepsilon,1}) - \sigma'(\chi_{\varepsilon,2})] \chi_\varepsilon} \leq c_L \int_{Q_t}{\left| \chi_\varepsilon \right|^2} = c_L \int_0^t{\left\| \chi_\varepsilon(s) \right\|^2_H ds}.
\end{equation}

The boundary term can be easily handled
\begin{equation}
\int_{\Sigma_t}{(1 \ast \alpha \vartheta_\varepsilon) \vartheta_\varepsilon} \geq \frac{\overline{\alpha}}{2} \left\| 1 \ast \vartheta_\varepsilon (t) \right\|^2_{L^2(\Gamma)}.
\end{equation}

Regarding the first term on the right hand side of the equality, we have
\begin{gather}
\int_{Q_t}{(1 \ast w + \eta_{0,\varepsilon}) \vartheta_\varepsilon} = \int_{\Omega}{(1 \ast g(t) + \eta_{0,_\varepsilon}) (1 \ast \vartheta_\varepsilon(t))} + \int_{\Gamma}{(1 \ast h(t)) (1 \ast \vartheta_\varepsilon(t))} \nonumber \\
- \int_{Q_t}{g (1 \ast \vartheta_\varepsilon)} - \int_{\Sigma_t}{h (1 \ast \vartheta_\varepsilon)} 
\leq c_\delta \left( \left\|1 \ast g (t) \right\|^2_H + \left\| \eta_{0,\varepsilon} \right\|^2_H + \left\|1 \ast h(t) \right\|^2_{L^2(\Gamma)}  \right) \nonumber \\
+ \delta \left\| 1 \ast \vartheta_\varepsilon (t) \right\|^2_V 
+ \int_0^t{ \left( \left\| g(s) \right\|_H + \left\| h(s) \right\|_{L^2(\Gamma)} \right) \left\| 1 \ast \vartheta_\varepsilon (s) \right\|_V}, 
\end{gather}
for all $\delta >0$, thanks to Young and H$\ddot{\text{o}}$lder inequalities and Trace Theorem; we recall that $g \in L^2(0,T;H)$ and $h \in L^2(0,T;L^{2}(\Gamma))$.  Moreover, using Young Theorem (see Reference \cite[Chap. 8, Prop. 8.9]{folland}) with $r = \infty$ and $p=q=2$, we have
\begin{gather}
\left\|1 \ast g (t) \right\|^2_H + \left\| \eta_{0,\varepsilon} \right\|^2_H + \left\|1 \ast h(t) \right\|^2_{L^2(\Gamma)}  
\leq \left\|1 \ast g \right\|^2_{L^{\infty}(0,T;H)} + \left\| \eta_{0,\varepsilon} \right\|^2_H  
 \nonumber \\
+ \left\|1 \ast h \right\|^2_{L^{\infty}(0,T;L^2(\Gamma))}
\leq c \left[ \left\| g \right\|^2_{L^2(0,T;H)} + \left\| \eta_{0,\varepsilon} \right\|^2_H + \left\|h \right\|^2_{L^2(0,T;L^2(\Gamma))}\right].
\end{gather}

Now, we choose $\delta >0$ small enough, so that we can apply Gronwall Lemma (see Reference \cite[Appendix, Lemma A.5]{brezis})
\begin{gather}
\int_{Q_t}{[\text{Log}_\varepsilon \vartheta_{\varepsilon,1} - \text{Log}_\varepsilon \vartheta_{\varepsilon,2}] \vartheta_\varepsilon} + \int_{Q_t}{ \left[ \beta'_\varepsilon(\chi_{\varepsilon,1}) - \beta'_\varepsilon(\chi_{\varepsilon,2}) \right] \chi_\varepsilon} 
+ \left\| \nabla \chi_\varepsilon \right\|^2_{L^2(Q_t)} \nonumber \\
+ \mu \left\| \chi_\varepsilon(t) \right\|^2_H  
+ \left\| 1 \ast \vartheta_\varepsilon (t) \right\|^2_V 
\leq c \left[ \left\| \eta_{0,\varepsilon} \right\|^2_H + \left\|g\right\|^2_{L^2(0,T;H)} + \left\|h\right\|^2_{L^2(0,T;L^2(\Gamma))} \right],
\end{gather}
$\forall \, t \in (0,T)$. This concludes the proof of Theorem \ref{teoeps}.

\subsubsection{Convergence of problems $(P_\varepsilon)$ to the original problem $(P_\mu)$ as $\varepsilon \searrow 0$}

Theorem \ref{teoeps} ensures that for $\varepsilon >0$ fixed there exists a solution to the approximated problem $(P_\varepsilon)$ and it is unique. Now, we want to get a solution to the original problem $(P_\mu)$ letting $\varepsilon$ tend to zero. We consider an arbitrary solution $(\vartheta_\varepsilon, \chi_\varepsilon)$ of $(P_\varepsilon)$ and we will perform a number of \textit{a priori} estimates so that we can take the limit as $\varepsilon \searrow 0$. In general, these estimates will hold for $\varepsilon$ small enough; however, they are independent on the parameter $\varepsilon$.

\paragraph{First a priori estimate.}
The first estimate is analogous to the first estimate performed in the calculations above. As before, we test the first equation (\ref{probleps1}) by $\vartheta_\varepsilon (t)$ and the second one (\ref{probleps2}) by $\partial _t \chi_\varepsilon (t)$; then, we integrate over the time interval $(0,t)$ and we sum the obtained equalities to each other. 

We can notice that all the terms are bounded by a constant $c$, independent of $\varepsilon$, except for the $\beta_\varepsilon$-term, which is estimated as follow
\begin{equation}
\int_{\Omega}{\beta_\varepsilon(\chi_{0,\mu})} \leq \int_{\Omega}{\beta(\chi_{0,\mu})} \leq c,
\end{equation}
due to the properties of Yosida regularization and to the assumption (\ref{chizerothetabound}).

Then, we have
\begin{gather}
\left\| I_\varepsilon(\vartheta_\varepsilon) \right\|_{L^{\infty}(0,T; L^1(\Omega))} + \left\| \vartheta_\varepsilon  \right\|^2_{L^2(0,T; V)} + \left\| \chi _\varepsilon \right\|^2_{H^1(0,T;H)} + \left\| \chi _\varepsilon \right\|^2_{L^{\infty}(0,T;V)} \nonumber \\
+ \left\| \beta_\varepsilon(\chi_\varepsilon) \right\|_{L^{\infty}(0,T; L^1(\Omega))} \leq c. 
\label{stima1eps}
\end{gather}
and, by comparison, $\left\| \partial_t \text{Log}_\varepsilon (\vartheta_\varepsilon) \right\|_{L^2(0,T; V')} \leq c$.

\paragraph{Second a priori estimate.}
The following estimate provides us a bound for the non linear term $\beta'_\varepsilon$. We test the equation (\ref{probleps2}) by $\beta'_\varepsilon (\chi_\varepsilon)$ and we integrate over $(0,t)$, $t \in (0,T]$
\begin{gather}
\mu \int_{\Omega}{\beta_\varepsilon(\chi_\varepsilon(t))} + \int_{Q_t}{\nabla \chi_\varepsilon \cdot \nabla \left[ \beta'_\varepsilon(\chi_\varepsilon) \right]} + \int_{Q_t}{\left| \beta'_\varepsilon(\chi_\varepsilon) \right|^2} 
= \int_{Q_t}{\vartheta_\varepsilon \beta'_\varepsilon(\chi_\varepsilon)} \nonumber \\
- \int_{Q_t}{\sigma'(\chi_\varepsilon) \beta'_\varepsilon(\chi_\varepsilon)} + \mu \int_{\Omega}{\beta_\varepsilon(\chi_{0,\mu})}. \label{eqstima2eps}
\end{gather}

The second term on the left hand side is non-negative, since $\beta'_\varepsilon$ is monotone; the right hand side is estimated in the following way
\begin{gather}
\int_{Q_t}{\vartheta_\varepsilon \beta'_\varepsilon(\chi_\varepsilon)} - \int_{Q_t}{\sigma'(\chi_\varepsilon) \beta'_\varepsilon(\chi_\varepsilon)} + \mu \int_{\Omega}{\beta_\varepsilon(\chi_0)} \nonumber \\
\leq \delta \left\| \beta'_\varepsilon(\chi_\varepsilon) \right\|^2_{L^2(0,t;H)}+ c_\delta \left( \left\| \vartheta_\varepsilon \right\|^2_{L^2(0,T;H)} +  \left\| \chi_\varepsilon \right\|^2_{L^2(0,T;H)} \right) + c
\end{gather}
for all $\delta >0$, due to H$\ddot{\text{o}}$lder and Young Inequalities, (\ref{betasigma}), (\ref{chizerothetabound}) and $0 \leq \beta_\varepsilon(r) \leq \beta(r)$,  $\forall \, r \in \textbf{R}$. We recall that the norms of $\vartheta_\varepsilon$ and $\chi_\varepsilon$ are uniformly bounded, thanks to (\ref{stima1eps}). 

Then, we choose $\delta >0 $ small enough and we take the supremum over $t \in (0,T)$:
\begin{equation}
\left\| \beta'_\varepsilon(\chi_\varepsilon) \right\|^2_{L^2(0,T;H)} \leq c. \label{stima2eps}
\end{equation}
Collecting all the estimates founded so far, we have the estimate
\begin{equation}
\left\| \chi_\varepsilon \right\|_{L^2(0,T;D(A;H))} \leq c.
\label{chiregular}
\end{equation}

\paragraph{Third a priori estimate.}
The last estimate we need is a bound for the logarithmic term in a suitable functional space. We recall that, by comparison in the first equation (\ref{probleps1}), we already know that $\text{Log}_\varepsilon(\vartheta_\varepsilon) \in H^1(0,T;V')$, with uniformly bounded norm (with respect to $\varepsilon$) thanks to (\ref{stima1eps}). 

In order to obtain this estimate, we have to be very careful in dealing with the boundary terms, so that we can obtain a suitable bound for the logarithm.

\begin{remark}
We recall that, since $\text{Log}_\varepsilon (r) \nearrow \log(r)$, $\forall \, r >0$, as $\varepsilon \rightarrow 0^+$ (see Reference \cite[Chap. II]{brezis}), then
\begin{equation}
\text{Log}_\varepsilon (r) \leq r, \ \ \forall \, r >1 \ \ \ \text{and} \ \ \ \text{Log}_\varepsilon (r) \leq 0, \ \ \forall \, r \leq 1.
\end{equation}
\label{osslog}
\end{remark}

We test (\ref{probleps1}) by $\text{Log}_\varepsilon(\vartheta_\varepsilon)$ and we integrate over $(0,t)$
\begin{gather}
\frac{1}{2} \left\| \text{Log}_\varepsilon \vartheta_\varepsilon (t) \right\|^2_H 
+ \int_{Q_t}{\nabla \vartheta_\varepsilon \cdot \nabla \left(\text{Log}_\varepsilon \vartheta_\varepsilon \right)} 
+ \int_{\Sigma_t}{\alpha \vartheta_\varepsilon \left( \text{Log}_\varepsilon \vartheta_\varepsilon \right)} \nonumber \\
= \int_{Q_t}{g \left( \text{Log}_\varepsilon \vartheta_\varepsilon \right)} 
+ \int_{\Sigma_t}{h \left(\text{Log}_\varepsilon \vartheta_\varepsilon \right)} 
- \int_{Q_t}{\left( \partial_t \chi_\varepsilon \right) \text{Log}_\varepsilon \vartheta_\varepsilon} 
+ \frac{1}{2} \left\| \text{Log}_\varepsilon \vartheta_{0,\varepsilon} \right\|^2_H.
\label{eqstima3eps}
\end{gather}

As $\text{Log}_\varepsilon$ is a monotone function, the second term is non negative; the boundary term needs a deeper analysis: keeping into account Remark \ref{osslog}, we have 
\begin{gather}
\int_{\Sigma_t}{\alpha \vartheta_\varepsilon \, \text{Log}_\varepsilon \vartheta_\varepsilon} 
\geq \overline{\alpha} \int_0^t{\left[ \int_{\Gamma \cap \left\{ \vartheta_\varepsilon \in (-\infty, 0) \cup (1, + \infty) \right\}}{ \vartheta_\varepsilon \, \text{Log}_\varepsilon \vartheta_\varepsilon} + \int_{\Gamma \cap \left\{ \vartheta_\varepsilon \in [0,1] \right\}}{\vartheta_\varepsilon \, \text{Log}_\varepsilon \vartheta_\varepsilon} \right]}  \nonumber \\
\geq \overline{\alpha} \int_0^t{\left[\int_{\Gamma \cap \left\{ \vartheta_\varepsilon \in [0,1] \right\}}{\vartheta_\varepsilon \text{Log}_\varepsilon \vartheta_\varepsilon} \right]} \geq -c,
\end{gather}
where the last inequality follows from a suitable application of Lebesgue Convergence Theorem.

Concerning the right hand side of the equality, the source terms are treated in the following way: the $g$-term can be easily handled together with the $\partial _t \chi_\varepsilon$-term
\begin{equation}
\int_{Q_t}{g \left( \text{Log}_\varepsilon \vartheta_\varepsilon \right)} - \int_{Q_t}{\left( \partial_t \chi_\varepsilon \right) \text{Log}_\varepsilon \vartheta_\varepsilon}
\leq \int_0^t{ \left( \left\| g(s) \right\|_H + \left\| \partial_t \chi_\varepsilon (s) \right\|_H  \right)  \left\| \text{Log}_\varepsilon \vartheta_\varepsilon (s)  \right\|_H};
\end{equation}
we recall that $g \in L^2(0,T;H)$ and $\partial_t \chi_\varepsilon \in L^2(0,T;H)$, with uniformly bounded norm, thanks to (\ref{stima1eps}). On the other hand, the $h$-term requires some additional calculus
\begin{gather}
\int_{\Sigma_t}{h \left(\text{Log}_\varepsilon \vartheta_\varepsilon\right)} \leq \int_0^t{\int_{\Gamma \cap \left\{ \vartheta_\varepsilon > 1 \right\}}{ h \vartheta_\varepsilon }} \leq \left\| \vartheta_\varepsilon \right\|^2_{L^2(0,T;L^2(\Gamma))} + \left\| h \right\|^2_{L^2(0,T;L^2(\Gamma))} \leq c,
\end{gather} 
recalling that $h \geq 0$ a.e. on $\Sigma$ and using Trace Theorem and (\ref{stima1eps}).

Regarding the last term, provided that $\varepsilon$ is small enough (say $\varepsilon \in (0,1)$), we have the following bound
\begin{equation}
\left\| \text{Log}_\varepsilon \vartheta_{0,\varepsilon} \right\|_{L^{\infty}(\Omega)} \leq c \left\| \vartheta_{0,\varepsilon} \right\|_{L^{\infty}(\Omega)} \leq c.
\label{logepsbound}
\end{equation}
using properties stated in Remark \ref{osslog} and thanks to (\ref{theta0eps}).

At this point, we can apply Gronwall Lemma (see Reference \cite[Appendix, Lemma A.5]{brezis}) and take the supremum over $t \in (0,T)$
\begin{equation}
\left\| \text{Log}_\varepsilon \vartheta_\varepsilon \right\|_{L^{\infty}(0,T;H)} \leq c.
\end{equation}

\paragraph{Passage to the limit as $\varepsilon \searrow 0$.}
Collecting all the previous estimates and using suitable compactness results, we can state that there exist
\begin{equation}
\begin{split}
&\vartheta_\mu \in L^2(0,T;V)  \\
&\chi_\mu \in L^2(0,T;D(A;H)) \cap H^1(0,T; H)  \\
&\xi_\mu \in L^2(0,T;H)  \\
&\mathscr{L} \in L^{\infty}(0,T;H) \cap H^1(0,T;V')
\end{split}
\end{equation}
such that, at least for a subsequence $\varepsilon_n \searrow 0$, they are limit of the approximating solutions
\begin{align}
&\vartheta_\varepsilon \rightharpoonup \vartheta_\mu \ \ \ \text{in} \ L^2(0,T;V) \\
&\chi_\varepsilon \rightharpoonup \chi_\mu \ \ \ \text{in} \ L^2(0,T;D(A;H)) \cap H^1(0,T; H) \\
&\beta'_\varepsilon(\chi_\varepsilon) \rightharpoonup \xi_\mu \ \ \ \text{in} \ L^2(0,T;H) \\
&\text{Log}_\varepsilon(\vartheta_\varepsilon) \stackrel{*}{\rightharpoonup} \mathscr{L} \ \ \ \text{in} \ L^{\infty}(0,T;H) \cap H^1(0,T;V').
\end{align}

It is clear that the Cauchy conditions (\ref{ultimaPmu}) are fulfilled. Indeed, weak convergence in $H^1(0,T;X)$ of a sequence implies weak convergence in $X$ (any functional space) of the corresponding initial data. Thus we can directly conclude $\chi_\mu(0) = \chi_{0,\mu}$ in $H$ and a.e. in $\Omega$. 

For the logarithmic term, we have $\text{Log}_\varepsilon \vartheta_\varepsilon (0) = \text{Log}_\varepsilon \vartheta_{0,\varepsilon} \rightharpoonup \mathscr{L}(0)$ in $V'$; on the other hand, $\text{Log}_\varepsilon \vartheta_{0,\varepsilon} \rightarrow \log \vartheta_0$ for a.e. $x \in \Omega$ (hence, it converges in measure) and $\text{Log}_\varepsilon \vartheta_{0,\varepsilon} \in L^{\infty}(\Omega)$, with uniformly bounded norm. Then, $\forall \, q \in [1, +\infty)$ we have (see Reference \cite[Chap. XI, Proposition 3.10]{visintinmodel})
\begin{equation}
\text{Log}_\varepsilon \vartheta_{0,\varepsilon} \xrightarrow{\varepsilon \rightarrow 0} \log \vartheta_0 \ \ \ \text{in} \ L^q(\Omega);
\end{equation}
in particular, $\text{Log}_\varepsilon \vartheta_{0,\varepsilon} \rightarrow \log \vartheta_0$ in $L^2(\Omega)$. Then, by uniqueness of limit, $\mathscr{L}(0) = \log \vartheta_0$ in $H$ and a.e. in $\Omega$.

Now, we deal with the non linear terms. Thanks to Aubin Lemma (see \cite[Chap. 1.5, Theorem 5.1]{aubin}), $\chi_\varepsilon \rightarrow \chi_\mu$ in $L^2(0,T;V)$, as $\varepsilon \searrow 0$, then $\sigma'(\chi_\varepsilon) \rightarrow \sigma'(\chi_\mu)$ in $L^2(0,T;H)$, since $\sigma'$ is Lipschitz. As far as the non linearity associated with the maximal monotone graph $\partial \beta$ is concerned, we can apply Barbu Lemma (see \cite[Chap. II, Lemma 1.3]{barbu}) and conclude that $\chi_\mu \in D(\partial \beta)$ and $\xi_\mu \in \partial \beta(\chi_\mu)$ a.e. in $Q$. Finally, concerning the logarithmic term, thanks to the compact embedding $H \subseteq V'$ and Aubin Lemma (see \cite[Chap. 1.5, Theorem 5.1]{aubin}), we have

\begin{equation}
\text{Log}_\varepsilon \vartheta_\varepsilon \rightarrow \mathscr{L} \ \ \  \text{*-weakly in} \ L^{\infty}(0,T;H) \ \text{and strongly in} \ L^2(0,T;V');
\end{equation}
thus, $\log_\varepsilon(\vartheta_\varepsilon) \rightarrow \mathscr{L}$ weakly in $L^{2}(0,T;H)$ and strongly in $L^2(0,T;V')$, as $\varepsilon \searrow 0$, since $\varepsilon \vartheta_\varepsilon \rightarrow 0$ strongly in $L^2(0,T;V)$. Therefore, the following limit holds
\begin{equation}
\lim_{\varepsilon \rightarrow 0} \int_0^T {\langle \log_\varepsilon(\vartheta_\varepsilon), \vartheta_\varepsilon \rangle} = \int_0^T {\langle \mathscr{L}, \vartheta_\mu \rangle};
\end{equation}
where $\langle \cdot, \cdot \rangle$ stands for the duality pairing between $V'$ and $V$; we apply Barbu Lemma once more and we conclude $\vartheta_\mu \in D(\log)$ (then, $\vartheta_\mu > 0$) and $\mathscr{L} = \log \vartheta_\mu$ a.e. in $Q$.

\begin{remark}
We can notice that we could have taken two different parameters $\varepsilon$ and $\varepsilon'$ in approximating $\log$ and $\partial \beta$, in order to prove Theorem \ref{teo1}. 

Moreover, we could have kept fixed either of them, say $\varepsilon$, and let $\varepsilon'$ tend to zero. This leads to an existence result for a semi-regularized problem. All the \textit{a priori} estimates are conserved in the limit (thanks to the semicontinuity of the norms), so that it is possible to let $\varepsilon$ tend to zero. The same can be done by switching the parameters.
\end{remark}

\paragraph{Uniqueness of solution to the problem $(P_\mu)$ and continuous dependence on data.}
We conclude the proof of Theorem \ref{teo1}, by showing uniqueness and continuous dependence of solutions from data. We follow the same procedure used previously to prove uniqueness of solution to problem $(P_\varepsilon)$. 

Let consider the first equation (\ref{1eqPmu}), formally integrated in time, and the second equation (\ref{2eqPmu}) of problem $(P_\mu)$ for two solutions $(\vartheta_{\mu,i}, \chi_{\mu,i}, \xi_{\mu,i})$ corresponding to the sets of data $(w_i, \vartheta_{0,i}, \chi_{0,\mu,i})$, $i=1,2$. 

We test the difference of the first equations by $\vartheta_\mu := \vartheta_{\mu,1} - \vartheta_{\mu,2}$ and the difference of the second ones by $ \chi_\mu := \chi_{\mu,1} - \chi_{\mu,2}$; then, we integrate over $(0,t)$ and we sum the obtained equalities
\begin{gather}
\int_{Q_t}{[\log \vartheta_{\mu,1} - \log \vartheta_{\mu,2}] \vartheta_\mu} 
+ \int_{Q_t}{(1 \ast \nabla \vartheta_\mu) \nabla \vartheta_\mu} 
+ \int_{\Sigma_t}{(1 \ast \alpha \vartheta_\mu) \vartheta_\mu}
+ \mu \int_{Q_t}{(\partial_t \chi_\mu)  \chi_\mu}  \nonumber \\
+ \int_{Q_t}{\left|\nabla \chi_\mu \right|^2} 
+ \int_{Q_t}{\xi_\mu \chi_\mu}
= \int_{Q_t}{(1 \ast w + \eta_{0,\mu}) \vartheta_\mu} - \int_{Q_t}{[\sigma'(\chi_{\mu,1}) - \sigma'(\chi_{\mu,2})] \chi_\mu} . \label{unicita}
\end{gather}
where we set again a similar notation for all the differences involved: $w := w_1 - w_2$, $\xi_\mu := \xi_{\mu,1} - \xi_{\mu,2}$, $\eta_{0,\mu} := [\log \vartheta_{0,1} -  \log \vartheta_{0,2}] + [\chi_{0,\mu,1} - \chi_{0,\mu,2}]$. 

After appropriate estimates, we get the following 
\begin{gather}
\int_{Q_t}{[\log \vartheta_{\mu,1} - \log \vartheta_{\mu,2}] \vartheta_\mu} + \int_{Q_t}{\xi_\mu \chi_\mu} 
+ \left\| \nabla \chi_\mu \right\|^2_{L^2(Q_t)} + \mu \left\| \chi_\mu(t) \right\|^2_H  
+ \left\| 1 \ast \vartheta_\mu (t) \right\|^2_V \nonumber \\
\leq M \left( \left\| \eta_{0,\mu} \right\|_H +  \left\|g\right\|_{L^2(0,T;H)}  + \left\|h\right\|_{L^2(0,T;L^2(\Gamma))}\right),
\end{gather}
with $M = M( \Omega, T)$ and $g:= g_1 - g_2$, $h := h_1 - h_2$.

From this relation, continuous dependence on data follows immediately and uniqueness of solution too (set $g_1 = g_2$, $h_1=h_2$ and $\eta_{0,\mu,1} = \eta_{0,\mu,2}$).

\section{Asymptotic behavior as $\mu \searrow 0$}
\label{mu}

In this section we prove Theorem \ref{teo3} and study the asymptotic behavior of the problem $(P_\mu)$ as the time relaxation parameter $\mu \searrow 0$.

The procedure is the following. We first prove some \textit{a priori} estimates, uniform with respect to $\mu$, and then let $\mu$ tends to zero. Using suitable compactness results, we find a subsequence of solutions to problem $(P_\mu)$ which (weakly) converges to a solution to the limit problem $(P_0)$; the whole family of solutions indeed converges, thanks to the uniqueness result stated in Theorem \ref{teo4}, which is proved at the end of this section. 

We make a distinction between the two cases (Hypothesis 1 or 2), since the passage to the limit as $\mu \searrow 0$ and the proof of existence of solutions to the limit problem use different methods in either case. 

First of all, let point out that all the \textit{a priori} estimates, which we perform, are formal: we should come back to the procedure used in \cite{global} and in the previous section, where problem $(P_\mu)$, $\mu >0$, has been solved by letting $\varepsilon \searrow 0$ in the approximating problem $(P_\varepsilon)$. 
However, in order not to make the exposition too heavy, we prefer to proceed formally. 

All the constants that appear in the following estimates are positive quantities, independent of the parameter $\mu$. 

\subsection{Uniform estimates}
\subsubsection{First a priori estimate}
The first estimate gives a uniform bound for the corresponding norms of solutions $(\vartheta_\mu, \chi_\mu)$. 

Assume either Hypothesis 1 or 2. We formally multiply equation (\ref{1eqPmu}) by $\vartheta_\mu$ and equation (\ref{2eqPmu}) by $\partial _t \chi_\mu$, then we integrate over $(0,t)$, with $t \in (0,T]$, and we sum the resulting equations. After some rearrangements, we have

\begin{gather}
\int_{\Omega}{\vartheta_\mu (t)} + 
\int_{Q_t}{\left| \nabla \vartheta_\mu \right|^2} + 
\int_{\Sigma_t}{\alpha \left|\vartheta_\mu\right|^2} +
\mu \int_{Q_t}{\left| \partial_t\chi_\mu \right|^2} + 
\frac{1}{2} \int_{\Omega}{\left| \nabla \chi_\mu (t) \right|^2}  
+ \int_{Q_t}{\xi_\mu \partial_t\chi_\mu} \nonumber \\
+ \int_{Q_t}{\sigma'(\chi_\mu)\partial_t \chi_\mu} 
= \int_{\Omega}{\vartheta_0} + \int_{Q_t}{w \vartheta_\mu} 
+ \frac{1}{2} \int_{\Omega}{\left| \nabla \chi_{0,\mu} \right|^2}.
\label{eqstima1mu}
\end{gather}

Let deal with each term separately. The first term on the left hand side is non negative, since $\vartheta_\mu >0$ a.e. on $Q$; the second and third terms are estimated from below by the norm of $\vartheta_\mu$ in $L^2(0,t;V)$. 

\paragraph{}
For the last two terms of the left hand side, we distinguish the cases in which Hypothesis 1 or 2 holds. Under Hypothesis 1, using $\sigma' = a$ constant and (\ref{betaquadr}), we have
\begin{gather}
\int_{Q_t}{\xi_\mu \partial_t\chi_\mu} + \int_{Q_t}{\sigma'(\chi_\mu)\partial_t \chi_\mu} = \int_{\Omega}{\beta(\chi_\mu (t))} + a \int_{\Omega}{\chi_\mu (t)} - \int_{\Omega}{\beta(\chi_{0,\mu})} - a \int_{\Omega}{\chi_{0,\mu}} \nonumber \\
\geq c_1 \left\| \chi_\mu (t) \right\|^2_H - c_2 + a \int_{\Omega}{\chi_\mu (t)} - \int_{\Omega}{\beta(\chi_{0,\mu})} - a \int_{\Omega}{\chi_{0,\mu}};
\end{gather}
we move the last four terms on the right hand side of the inequality (\ref{eqstima1mu}) and, thanks to (\ref{succchi0mu}) and Young Inequality, we estimate them this way
\begin{equation}
-a \int_{\Omega}{\chi_\mu (t)} + \int_{\Omega}{\beta(\chi_{0,\mu})} + a \int_{\Omega}{\chi_{0,\mu}} \leq \delta \left\| \chi_\mu (t) \right\|^2_H + c_\delta \ \ \ \forall \, \delta >0.
\end{equation}

\paragraph{}
Under Hypothesis 2, using (\ref{betasigmaquadr}), we have
\begin{gather}
\int_{Q_t}{\xi_\mu \partial_t\chi_\mu} + \int_{Q_t}{\sigma'(\chi_\mu)\partial_t \chi_\mu} = \int_{\Omega}{\beta(\chi_\mu (t))} + \int_{\Omega}{\sigma(\chi_\mu (t))} - \int_{\Omega}{\beta(\chi_{0,\mu})} \nonumber \\
-  \int_{\Omega}{\sigma(\chi_{0,\mu})} 
\geq c_1 \left\| \chi_\mu (t) \right\|^2_H - c_2 - \int_{\Omega}{\beta(\chi_{0,\mu})} -  \int_{\Omega}{\sigma(\chi_{0,\mu})},
\end{gather}
where the last two terms are uniformly bounded, thanks to (\ref{stimasigma}) and (\ref{succchi0mu}). 

\paragraph{}
Regarding the right hand side of equality (\ref{eqstima1mu}), the last term is uniformly bounded, thanks to (\ref{succchi0mu}), while the $w$-term is bounded from above by the following norms
\begin{equation}
\begin{split}
&\int_{Q_t}{g \vartheta_\mu} \leq \left\| g\right\|_{L^2(Q_t)} \left\| \vartheta_\mu \right\|_{L^2(Q_t)} \leq \delta \left\| \vartheta_\mu \right\|^2_{L^2(Q_t)} + c_\delta \left\| g\right\|^2_{L^2(Q_t)}  \\
&\int_{\Sigma_t}{h \vartheta_\mu} \leq \delta \left\| \vartheta_\mu \right\|^2_{L^2(\Sigma_t)} + c_\delta \left\| h\right\|^2_{L^2(\Sigma_t)} \leq \delta  \left\| \vartheta_\mu \right\|^2_{L^2(0,t; V)} + c_\delta \left\| h\right\|^2_{L^2(\Sigma_t)},
\end{split}
\end{equation}
using H$\ddot{\text{o}}$lder and Young Inequalities and the Trace Theorem.

Then, choosing $\delta > 0$ sufficiently small and taking the supremum on $[0,T]$, we conclude
\begin{equation}
\left\| \vartheta_\mu  \right\|^2_{L^2(0,T; V)} + \mu \left\| \partial_t \chi_\mu \right\|^2_{L^2(0,T;H)}+ \left\| \chi _\mu \right\|^2_{L^{\infty}(0,T;V)} \leq c; 
\label{stima1cap7}
\end{equation}
moreover, by comparison, we have $\left\| \log \vartheta_\mu \right\|_{L^{\infty}(0,T;V')} \leq c$.

\subsubsection{Second a priori estimate}
In order to obtain a uniform estimate for the norm of the non linear term $\xi_\mu \in \partial \beta (\chi_\mu)$, let formally multiply the equation (\ref{2eqPmu}) by $\xi_\mu$ and integrate it over $(0,t)$, $t \in (0,T]$; after some rearrangements, we get
\begin{equation}
\mu \int_{\Omega}{\beta(\chi_\mu(t))}+ \int_{Q_t}{ \beta''(\chi_\mu) \left| \nabla \chi_\mu \right|^2} + \int_{Q_t}{\left| \xi_\mu \right|^2} 
=  \mu \int_{\Omega}{\beta_(\chi_{0,\mu})} + \int_{Q_t}{\vartheta_\mu \xi_\mu} - \int_{Q_t}{\sigma'(\chi_\mu) \xi_\mu}.
\end{equation}
The first term on the right hand side of this equality is uniformly bounded, thanks to (\ref{succchi0mu}), while we can easily handle the other two terms applying H$\ddot{\text{o}}$lder and Young Inequalities
\begin{equation}
\begin{split}
\int_{Q_t}{\vartheta_\mu \xi_\mu} &\leq \delta \left\| \xi_\mu \right\|^2_{L^2(0,t;H)} + c_\delta \left\| \vartheta_\mu \right\|^2_{L^2(0,T;H)}  \\
\int_{Q_t}{\sigma'(\chi_\mu) \xi_\mu} &\leq  \delta \left\| \xi_\mu \right\|^2_{L^2(0,t;H)} + c_\delta \left\| \chi_\mu \right\|^2_{L^2(0,T;H)},
\end{split}
\end{equation}
for each arbitrary $ \delta > 0$; we recall that $\sigma'$ is a Lipschitz function (see (\ref{betasigma})) and that the norms of $\vartheta_\mu$ and $\chi_\mu$ are uniformly bounded with respect to $\mu$ thanks to estimate (\ref{stima1cap7}). 

We choose $\delta >0$ small enough and we take the supremum over $t \in (0,T)$; finally, we have in particular 
\begin{equation}
\left\| \xi_\mu \right\|^2_{L^2(0,T;H)} \leq c;
\end{equation}
moreover, by comparison in the equation (\ref{2eqPmu}) and by Elliptic Regularity Theorem, we can conclude $\left\| \chi_\mu \right\|_{L^2(0,T;D(A;H))} \leq c$.

\subsubsection{Cauchy estimate}
Consider two problems $(P_\mu)$ and $(P_\nu)$, with $\mu >\nu >0$; repeating the same arguments as above to show the uniqueness of solutions to problem $(P_\mu)$ ( Section \ref{exist}), subtract the first two equations (\ref{1eqPmu}) integrated in time and the second two equations (\ref{2eqPmu}). 

Let multiply the resulting expression by $\vartheta_\mu - \vartheta_\nu$ and by $\chi_\mu - \chi_\nu$ respectively; finally, we integrate over the interval $(0,t)$ and we sum the equations
\begin{gather}
\int_{Q_t}{[\log \vartheta_\mu - \log \vartheta_\nu] (\vartheta_\mu - \vartheta_\nu)} 
+ \int_{Q_t}{(1 \ast \nabla (\vartheta_\mu - \vartheta_\nu)) \nabla (\vartheta_\mu - \vartheta_\nu)}  \nonumber \\
+ \int_{\Sigma_t}{ \left( 1 \ast \alpha (\vartheta_\mu - \vartheta_\nu) \right) (\vartheta_\mu - \vartheta_\nu)} 
+ \lambda \int_{Q_t}{  \partial_t (\chi_\mu - \chi_\nu) (\chi_\mu - \chi_\nu)}
+ \int_{Q_t}{\left|\nabla (\chi_\mu - \chi_\nu) \right|^2}  \nonumber \\
+ \int_{Q_t}{(\xi_\mu - \xi_\nu) (\chi_\mu - \chi_\nu)} 
+ \int_{Q_t}{[\sigma'(\chi_\mu) - \sigma'(\chi_\nu)] (\chi_\mu - \chi_\nu)} \nonumber \\
\leq \int_{Q_t}{(\chi_{0,\mu} - \chi_{0,\nu}) (\vartheta_\mu - \vartheta_\nu)},
\end{gather}
for some $\nu < \lambda < \mu $.

At this point, we distinguish the case where Hypothesis 1 or 2 holds. 

\paragraph{}
Under Hypothesis 1, $\sigma'=a \in \textbf{R}$ constant, which implies that the term involving $\sigma'$ is identically zero. 

After some manipulations and estimates, we get the following inequality
\begin{gather}
\int_{Q_t}{[\log \vartheta_\mu - \log \vartheta_\nu] (\vartheta_\mu - \vartheta_\nu)}
+ \int_{Q_t}{(\xi_\mu - \xi_\nu) (\chi_\mu - \chi_\nu)} \nonumber \\
+ \left\| 1 \ast (\vartheta_\mu - \vartheta_\nu) (t) \right\|^2_V + \left\| \nabla \chi_\mu - \nabla \chi_\nu \right\|^2_{L^2(Q_t)} \leq c \left\| \chi_{0,\mu} - \chi_{0,\nu} \right\|^2_H,
\label{cauchyhp1}
\end{gather}
$\forall \, t \in (0,T)$.

Thanks to this estimate and to (\ref{succchi0mu}), we can conclude that $\{ 1 \ast \vartheta_\mu \}$ and $\{ \nabla \chi_\mu \}$ are Cauchy sequence in $L^{\infty}(0,T;V)$ and $L^2(0,T;H)$ respectively.

\paragraph{}
Under Hypothesis 2, using the strong monotonicity property (\ref{betasigmahp2}) assumed for $\partial \beta + \sigma'$, we have
\begin{gather}
\int_{Q_t}{\left|\nabla (\chi_\mu - \chi_\nu) \right|^2} + \int_{Q_t}{(\xi_\mu - \xi_\nu) (\chi_\mu - \chi_\nu)} + \int_{Q_t}{[ \sigma'(\chi_\mu) - \sigma'(\chi_\nu) ] ( \chi_\mu - \chi_\nu )}  \nonumber \\
\geq \left\| \nabla(\chi_\mu - \chi_\nu) \right\|^2_{L^2(0,T;H)} + \rho \left\| \chi_\mu - \chi_\nu \right\|^2_{L^2(0,T;H)} \geq c \left\| \chi_\mu - \chi_\nu \right\|^2_{L^2(0,T;V)}.
\end{gather}

Then, after some estimates, we get
\begin{gather}
\int_{Q_t}{[\log \vartheta_\mu - \log \vartheta_\nu] (\vartheta_\mu - \vartheta_\nu)}
+ \left\| 1 \ast (\vartheta_\mu - \vartheta_\nu) \right\|^2_{L^{\infty}(0,T;V)} + \left\|  \chi_\mu -  \chi_\nu \right\|^2_{L^2(0,T;V)} \nonumber \\
\leq c \left\| \chi_{0,\mu} - \chi_{0,\nu} \right\|^2_H.
\label{cauchyhp2}
\end{gather}
and we conclude that $\{ 1 \ast \vartheta_\mu \}$ and $\{ \chi_\mu \}$ are Cauchy sequence in $L^{\infty}(0,T;V)$ and $L^2(0,T;V)$ respectively.

\subsection{Passage to the limit as $\mu \searrow 0$ under Hypothesis 1}
Using the previous estimates, we can state that there exist
\begin{equation}
\begin{split}
&\vartheta \in L^2(0,T;V)  \\
&\chi \in L^2(0,T;D(A;H))  \\
&\xi \in L^2(0,T;H)  \\
&\zeta \in L^{\infty}(0,T;V')
\end{split}
\end{equation}
such that the following convergences hold (at least for a subsequence)
\begin{eqnarray}
\vartheta_\mu &\rightharpoonup& \vartheta \ \ \ \text{in} \ L^2(0,T;V) \label{convthetamu}\\
1\ast \vartheta_\mu &\rightarrow& 1 \ast \vartheta \ \ \ \text{in} \ L^{\infty}(0,T;V) \label{conv1starthetamu}\\
\chi_\mu &\rightharpoonup& \chi \ \ \ \text{in} \ L^2(0,T;D(A;H)) \label{convchimu}\\
\mu \partial_t \chi_\mu &\rightharpoonup& 0 \ \ \ \text{in} \ L^2(0,T;H)\\
\xi_\mu &\rightharpoonup& \xi \ \ \ \text{in} \ L^2(0,T;H)\\
\log(\vartheta_\mu) &\stackrel{*}{\rightharpoonup}& \zeta \ \ \ \text{in} \ L^{\infty}(0,T;V').
\end{eqnarray}

Now, we deal with the non linear terms. We mainly refer to \cite[Sec. 3.3]{convergence}, where an asymptotic behavior as the interfacial energy coefficient tends to $0$ is studied for the same PDE system with different boundary conditions. 

In particular, in order to analyze the term with the maximal monotone graph $\partial \beta$, we perform an \textit{ad hoc} weighted estimate, choosing $t^{2\gamma}$ as weight, with $\gamma = 3/4$. 
The choose of such a weight will be clear in some later inequalities (see (\ref{disug1})), which happen to be fulfilled if $1/2 < \gamma < 1$, and we take $\gamma = 3/4$ at once.

\begin{lemma}[see Lemma 3.1, \cite{convergence}]
Let $\gamma >0$, then the following uniform estimate holds
\begin{equation}
\int_{Q}{t^{2\gamma} \frac{\left| \partial_t \vartheta_\mu \right|^2}{\vartheta_\mu}} + \sup_{t \in (0,T)}{ t^{2\gamma} \left\| \vartheta_\mu (t) \right\|^2_V } 
+ \mu \sup_{t \in (0,T)}{ t^{2\gamma}\int_{\Omega}{\left| \partial_t \chi_\mu (t) \right|^2} } + \int_{Q}{t^{2\gamma}\left| \nabla \partial_t \chi_\mu \right|^2} \leq c.
\end{equation}
\label{stimat}
\end{lemma}

\begin{proof}
We formally test the equation (\ref{1eqPmu}) by $t^{2\gamma} \partial_t \vartheta_\mu$; next, we formally differentiate (\ref{2eqPmu}) with respect to time (we obtain a second order equation) and we test it by $t^{2\gamma} \partial_t \chi_\mu$. 

We notice that, since $\sigma' = a$, then $\sigma'' =0$. Finally, we add the equalities to each other and we integrate over $(0,t)$:
\begin{gather}
\int_{Q_t}{\partial_t \log \vartheta_\mu \, t^{2\gamma} \, \partial_t \vartheta_\mu}
+ \int_{Q_t}{ \nabla \vartheta_\mu  t^{2\gamma} \nabla (\partial_t \vartheta_\mu)}
+ \int_{\Sigma_t}{\alpha \vartheta_\mu t^{2\gamma} \, \partial_t \vartheta_\mu} 
+ \mu \int_{Q_t}{ t^{2\gamma} \partial_t \chi_\mu \, \partial^2_t \chi_\mu} \nonumber \\
+ \int_{Q_t}{t^{2\gamma} \left| \nabla (\partial_t \chi_\mu) \right|^2} +
\int_{Q_t}{t^{2\gamma} \xi'_\mu \left|\partial_t \chi_\mu \right|^2}  
= \int_{Q_t}{g t^{2\gamma} \partial_t \vartheta_\mu} +
\int_{\Sigma_t}{h t^{2\gamma} \partial_t \vartheta_\mu}.
\end{gather}

We deal with each term separately.

The first and the last two terms are non negative, since (recalling that $\vartheta_\mu >0$)
\begin{equation}
\int_{Q_t}{(\partial_t \log \vartheta_\mu) t^{2\gamma} (\partial_t \vartheta_\mu)} = \int_{Q_t}{t^{2\gamma} \frac{\left| \partial_t \vartheta_\mu \right|^2}{\vartheta_\mu}} \geq 0,
\end{equation}
and $\partial \beta$ is monotone;next, we estimate the second one as follows: thanks to (\ref{stima1cap7}),
\begin{gather}
\int_{Q_t}{ \nabla \vartheta_\mu  t^{2\gamma} \nabla (\partial_t \vartheta_\mu)} = \frac{t^{2\gamma}}{2} \int_\Omega{\left| \nabla \vartheta_\mu (t) \right|^2} - \int_{Q_t}{ 2 \gamma t^{2\gamma -1} \left| \nabla \vartheta_\mu \right|^2} \nonumber \\
\geq \frac{t^{2\gamma}}{2} \left\| \nabla \vartheta_\mu (t) \right\|^2_{H} - 2 \gamma T^{2 \gamma -1} \left\| \nabla \vartheta_\mu \right\|^2_{L^2(0,T;H)} \geq \frac{t^{2\gamma}}{2} \left\| \nabla \vartheta_\mu (t) \right\|^2_{H} - c. 
\end{gather}

Similarly, using (\ref{stima1cap7}), we have 
\begin{gather}
\int_{\Sigma_t}{ t^{2\gamma} \alpha \vartheta_\mu (\partial_t \vartheta_\mu)} 
= \frac{t^{2\gamma}}{2} \int_\Gamma{ \alpha \left| \vartheta_\mu (t) \right|^2} - \int_{\Sigma_t}{ 2 \gamma t^{2\gamma -1} \alpha \left| \vartheta_\mu \right|^2} \nonumber \\
\geq \overline{\alpha} \frac{t^{2\gamma}}{2} \left\| \vartheta_\mu (t) \right\|^2_{L^2(\Gamma)} - 2 \overline{\alpha} \gamma T^{2 \gamma -1} \left\| \vartheta_\mu \right\|^2_{L^2(0,T;L^2(\Gamma))}\geq \overline{\alpha} \frac{t^{2\gamma}}{2} \left\| \vartheta_\mu (t) \right\|^2_{L^2(\Gamma)} - c;
\end{gather}
and
\begin{gather}
\mu \int_{Q_t}{ t^{2\gamma} (\partial_t \chi_\mu) (\partial^2_t \chi_\mu)} = 
\frac{\mu t^{2\gamma}}{2} \int_\Omega{\left| \partial_t \chi_\mu (t) \right|^2} - \mu \int_{Q_t}{ 2 \gamma t^{2\gamma -1} \left| \partial_t \chi_\mu \right|^2} \nonumber \\
\geq \frac{\mu t^{2\gamma}}{2} \left\| \partial_t \chi_\mu (t) \right\|^2_{H} - 2\mu \gamma T^{2 \gamma -1} \left\| \partial_t \chi_\mu \right\|^2_{L^2(0,T;H)} 
\geq \frac{\mu t^{2\gamma}}{2} \left\| \partial_t \chi_\mu (t) \right\|^2_{H} - c.
\end{gather}

The source terms can be easily handled
\begin{gather}
\int_{Q_t}{g t^{2\gamma} \partial_t \vartheta_\mu} 
= t^{2\gamma} \int_{\Omega}{g (t) \vartheta_\mu (t)} - \int_{Q_t}{2 \gamma t^{2\gamma -1} \vartheta_\mu g} - \int_{Q_t}{t^{2\gamma} \vartheta_\mu \partial_t g} \nonumber \\
\leq \delta t^{2\gamma} \left\| \vartheta_\mu (t) \right\|^2_H + c_\delta \left\|g(t) \right\|^2_H + cT^{2\gamma -1}\left[ \left\|\vartheta_\mu \right\|^2_{L^2(0,T;H)} + \left\| g \right\|^2_{H^1(0,T;H)} \right] \nonumber \\
\leq \delta t^{2\gamma} \left\| \vartheta_\mu (t) \right\|^2_V + c, 
\end{gather}
for each $\delta >0$; in the last inequality we used $g \in H^1(0,T;H) \subseteq C^0([0,T];H)$ and (\ref{stima1cap7}). 
The same calculations hold for the $h$ term, thanks to the fact that $h \in H^1(0,T;L^2(\Gamma)) \subseteq C^0([0,T];L^2(\Gamma))$, the Trace Theorem and (\ref{stima1cap7})
\begin{equation}
\int_{\Sigma_t}{h t^{2\gamma} \partial_t \vartheta_\mu} \leq \delta t^{2\gamma} \left\| \vartheta_\mu (t) \right\|^2_{L^2(\Gamma)} + c, 
\end{equation}
for each $\delta >0$.

Choosing $\delta >0$ small enough, we have
\begin{gather}
\int_{Q_t}{t^{2\gamma} \frac{\left| \partial_t \vartheta_\mu \right|^2}{\vartheta_\mu}} + t^{2\gamma}\left( \left\| \nabla \vartheta_\mu (t) \right\|^2_{H} + \left\| \vartheta_\mu (t) \right\|^2_{L^2(\Gamma)}+ \mu \left\| \partial \chi_\mu (t) \right\|^2_{H}  \right) \nonumber \\
+ \int_{Q_t}{t^{2\gamma} \left| \nabla (\partial_t \chi_\mu) \right|^2} \leq c,
\end{gather}
and we get the desired result by taking the supremum over $t \in (0,T)$.
\end{proof}

At this point, we use the identity
\begin{equation}
\partial_t ( t^\gamma \vartheta_\mu)= \gamma t^{\gamma -1}\vartheta_\mu + 2 t^\gamma \left( \partial_t \sqrt{\vartheta_\mu} \right) \cdot t^{\frac{\gamma}{2}} \sqrt{\vartheta_\mu} \cdot t^{-\frac{\gamma}{2}}
\end{equation} 
and observe that, thanks to Sobolev Inequalities and (\ref{stima1cap7}),
\begin{equation}
\left\| t^{\frac{\gamma}{2}} \sqrt{\vartheta_\mu} \right\|_{L^{\infty}(0,T;L^{12}(\Omega))} = \left\| t^{\gamma} \vartheta_\mu \right\|^{\frac{1}{2}}_{L^{\infty}(0,T;L^6(\Omega))} \leq c T^{\frac{\gamma}{2}} \left\| \vartheta_\mu \right\|^{\frac{1}{2}}_{L^2(0,T;V)} \leq c. 
\end{equation}

Thanks to the previous Lemma, we can say $\left\| t^\gamma \left( \partial_t \sqrt{\vartheta_\mu} \right)\right\|_{L^2(0,T;H)} \leq c$. Now, we look for exponents $p,q, r,s>1$ such that
\begin{gather}
\left\| t^\gamma \left( \partial_t \sqrt{\vartheta_\mu} \right) \cdot t^{\frac{\gamma}{2}} \sqrt{\vartheta_\mu} \cdot t^{-\frac{\gamma}{2}} \right\|_{L^q(0,T; L^{\frac{12}{7}}(\Omega))} \nonumber \\
\leq \left\| t^\gamma \left( \partial_t \sqrt{\vartheta_\mu} \right)\right\|_{L^2(0,T;H)} \left\| t^{\frac{\gamma}{2}} \sqrt{\vartheta_\mu} \right\|_{L^{\infty}(0,T;L^{12}(\Omega))} \left\| t^{-\frac{\gamma}{2}}  \right\|_{L^s(0,T)} \leq c, 
\end{gather}
and
\begin{gather}
\left\| t^{\gamma -1}\vartheta_\mu \right\|_{L^p(0,T;L^6(\Omega))} \leq \left\| t^{\gamma -1} \right\|_{L^r(0,T)} \left\| \vartheta_\mu \right\|_{L^2(0,T; L^6(\Omega))} \nonumber \\
\leq \left\| t^{\gamma -1} \right\|_{L^r(0,T)} \left\| \vartheta_\mu \right\|_{L^2(0,T;V)} \leq c,
\end{gather}
where we used H$\ddot{\text{o}}$lder and Sobolev Inequalities. Thus, we need the following constraints to be fulfilled:
\begin{equation}
\frac{1}{q} = \frac{1}{2} + \frac{1}{s}, \ \  \frac{1}{p} = \frac{1}{2} + \frac{1}{r},  \  \ t^{- \frac{\gamma}{2}} \in L^s(0,T), \ \ t^{\gamma -1} \in L^r(0,T).
\label{disug1}
\end{equation}
As $\gamma = 3/4$, we can take $p=6/5 , \, r=3, \, q=14/13, \, s=7/3$.
Then we have
\begin{gather}
\left\| \partial_t \left( t^{\gamma} \vartheta_\mu \right) \right\|_{L^{14 / 13}(0,T; L^{12/7}(\Omega))} \leq c.
\end{gather}

At this point, we can apply a standard compactness lemma (see Reference \cite[Sec. 8, Corollary 4]{simon}) and obtain $ t^{\frac{3}{4}} \left( \vartheta_\mu - \vartheta \right) \rightarrow 0$ strongly in $C^0([0,T];H)$; then, 
\begin{equation}
\vartheta_\mu \chi_\mu \rightarrow \vartheta \chi \ \ \ \text{weakly in} \ L^1(Q_t), \label{convthetachimu}
\end{equation}
since, $\vartheta_\mu \chi_\mu = t^{3/4} \vartheta_\mu t^{-3/4} \chi_\mu$, $\chi_\mu \rightharpoonup \chi$ in $L^{\infty}(0,T;V)$ and $t^{-3/4} \in L^1(0,T)$. 

Finally, collecting all the previous results, we have
\begin{equation}
\limsup_{\mu \rightarrow 0}{\int_{Q}{\xi_\mu \chi_\mu}} \leq \int_{Q}{ \xi \chi}.
\end{equation}
Indeed, by comparison in the second equation (\ref{2eqPmu}) of problem $(P_\mu)$, we have
\begin{align}
\int_{Q}{\xi_\mu \chi_\mu} &= - \mu \int_{Q}{\left(\partial_t \chi_\mu\right) \chi_\mu} - \int_{Q}{\left| \nabla \chi_\mu \right|^2} - \int_{Q}{ a \chi_\mu} + \int_{Q}{\vartheta_\mu \chi_\mu} \nonumber \\
&= - \frac{\mu}{2} \left\| \chi_\mu (T) \right\|^2_H + \frac{\mu}{2} \left\| \chi_\mu (0) \right\|^2_H - \left\| \nabla \chi_\mu \right\|^2_{L^2(Q)} - \int_Q{a\chi_\mu} + \int_{Q}{\vartheta_\mu \chi_\mu} \nonumber \\
&\leq \frac{\mu}{2} \left\| \chi_{0,\mu} \right\|^2_H - \left\| \nabla \chi_\mu \right\|^2_{L^2(Q)} - \int_Q{a \chi_\mu} + \int_{Q}{\vartheta_\mu \chi_\mu} \nonumber \\
&\leq \mu c - \left\| \nabla \chi_\mu \right\|^2_{L^2(Q)} - \int_Q{a \chi_\mu} + \int_{Q}{\vartheta_\mu \chi_\mu}.
\end{align}

On the other hand, thanks to (\ref{convchimu}), (\ref{convthetachimu})  and the linearity of $\sigma$, 
\begin{equation}
\int_{Q}{\xi \chi} = - \int_{Q}{\left| \nabla \chi \right|^2} - \int_Q{a\chi} + \int_{Q}{\vartheta \chi},
\end{equation}
hence
\begin{equation}
\limsup_{\mu \rightarrow 0}{\int_{Q}{\xi_\mu \chi_\mu}} \leq - \int_{Q}{\left| \nabla \chi \right|^2} - \int_Q{a\chi} + \int_{Q}{\vartheta \chi} = \int_{Q}{ \xi \chi}.
\end{equation}

In conclusion, owing to \cite[Sec. II, Lemma 1.3]{barbu}, we get $\chi \in D(\partial \beta)$ and $\xi \in \partial \beta(\chi) $ a.e. in $Q$.

Concerning the logarithmic term, we can follow the same procedure we have seen for the $\xi$ term. By comparison in the first equation integrated in time and due to the previous convergences (see (\ref{convthetamu}), (\ref{conv1starthetamu}), (\ref{convthetachimu})), we get
\begin{equation}
\limsup_{\mu \rightarrow 0}{\int_0^T{\langle \zeta_\mu, \vartheta_\mu \rangle}} \leq \int_0^T{\langle \zeta, \vartheta \rangle}, \label{limsuplog}
\end{equation}
where $\zeta_\mu = \log \vartheta_\mu$.
Hence, owing again to \cite[Sec. II, Lemma 1.3]{barbu}, and knowing that $\zeta_\mu \in \text{Log} \, \vartheta_\mu = \partial \Psi (\vartheta_\mu)$ (see \cite[Remark 4.3]{convergence}), we have $\vartheta \in D(\text{Log})$ and $\zeta \in \partial \Psi (\vartheta) = \text{Log} \, \vartheta$.

It is also possible to prove that the solution $\vartheta$ to problem $(P_0)$, which is the absolute temperature, is strictly positive (see \cite[Theorem 4.7]{convergence}). 

\subsection{Passage to the limit as $\mu \searrow 0$ under Hypothesis 2}

Taking into account the uniform estimates we have performed, we can conclude that there exist
\begin{equation}
\begin{split}
&\vartheta \in L^2(0,T;V)  \\
&\chi \in L^2(0,T;D(A;H))  \\
&\xi \in L^2(0,T;H)  \\
&\zeta \in L^{\infty}(0,T;V')
\end{split}
\end{equation}
such that the following convergences (as $\mu \searrow 0$) hold, at least for a subsequence,  
\begin{eqnarray}
\vartheta_\mu &\rightharpoonup& \vartheta \ \ \ \text{in} \ L^2(0,T;V)\\
1\ast \vartheta_\mu &\rightarrow& 1 \ast \vartheta \ \ \ \text{in} \ L^{\infty}(0,T;V) \\
\chi_\mu &\rightarrow& \chi \ \ \ \text{weakly in} \ L^2(0,T;D(A;H)) \nonumber \\ 
&&\ \ \ \ \ \text{strongly in} \ L^2(0,T;V) \\
\mu \partial_t \chi_\mu &\rightharpoonup& 0 \ \ \ \text{in} \ L^2(0,T;H)\\
\xi_\mu &\rightharpoonup& \xi \ \ \ \text{in} \ L^2(0,T;H)\\
\log(\vartheta_\mu) &\stackrel{*}{\rightharpoonup}& \zeta \ \ \ \text{in} \ L^{\infty}(0,T;V').
\end{eqnarray}

Thanks to these convergences, we can identify the $\sigma'$ term and we can immediately apply the result stated in \cite[Sec. II, Lemma 1.3]{barbu}, in order to get $ \chi \in D(\partial \beta)$ and $\vartheta \in D(\text{Log})$, $ \xi \in \partial \beta(\chi)$ and $\zeta \in \partial \Psi (\vartheta) = \text{Log} \, \vartheta$ a.e. in $Q$.

\subsection{Uniqueness of solution to problem $(P_0)$ and continuous dependence on data}

To show the uniqueness of solution and continuous dependence on data, we follow the same method used in Section \ref{exist}, to show uniqueness of solution to the problem $(P_\mu)$ with $\mu >0$. 

In particular, we distinguish the case where Hypothesis 1 or 2 holds, since the behavior of the non linear terms $\beta$ and $\sigma$ is quite different. 

\paragraph{}
Assuming Hypothesis 1, we write problem $(P_0)$ for two distinct solution $(\vartheta_i, \chi_i, \xi_i, \zeta_i)$, $i=1,2$, and we subtract one equation to the other.
\begin{align}
&[\zeta_1 - \zeta_2] + [\chi_1 - \chi_2] + 1 \ast B(\vartheta_1 - \vartheta_2) = 1 \ast (w_1 - w_2) + \eta_{0,1} - \eta_{0,2} \label{eq1diff0}\\
&A (\chi_1 - \chi_2) + [\xi_1 - \xi_2] + [\sigma'(\chi_1) - \sigma'(\chi_2)] = \vartheta_1 - \vartheta_2. \label{eq2diff0}
\end{align}
We test (\ref{eq1diff0}) by $\vartheta:= \vartheta_1 - \vartheta_2$ and (\ref{eq2diff0}) by $\chi:= \chi_1 - \chi_2$ and we integrate over $(0,t)$, $t \in (0,T]$; finally we add the resulting equation:
\begin{equation}
\int_{Q_t}{\zeta \vartheta} 
+ \int_{Q_t}{(1 \ast \nabla \vartheta) \nabla \vartheta}  
+ \int_{\Sigma_t}{ \left( 1 \ast \alpha \vartheta \right) \vartheta} 
+ \int_{Q_t}{\left|\nabla \chi \right|^2}  
+ \int_{Q_t}{\xi \chi}
= \int_{Q_t}{(1 \ast w + \eta_0) \vartheta},
\end{equation}
where $\zeta := \zeta_1 - \zeta_2$, $\xi:= \xi_1 - \xi_2$, $w:= w_1 - w_2$, $\eta_0:= \eta_{0,1} - \eta_{0,2}$ ; the $\sigma'$ term is identically zero, since $\sigma' = a$ constant.

After suitable estimates, similar to those performed in Section \ref{exist}, we get the following inequality $\forall \, t \in (0,T)$
\begin{gather}
\int_{Q_t}{\zeta \vartheta} + \int_{Q_t}{\xi \chi} 
+ \left\| \nabla \chi \right\|^2_{L^2(Q_t)}   
+ \left\| 1 \ast \vartheta (t) \right\|^2_V \nonumber \\
\leq \tilde{M}\left( \left\| \eta_0 \right\|^2_H +  \left\|g\right\|^2_{L^2(0,T;H)} + \left\|h\right\|^2_{L^2(0,T;L^2(\Gamma))} \right), 
\end{gather}
where $\tilde{M} = \tilde{M}(\Omega, T)$ and $g:= g_1 - g_2$, $h:= h_1 - h_2$.

\paragraph{}
Assuming Hypothesis 2 and using (\ref{betasigmahp2}), the terms related to $\partial \beta$ e $\sigma$ are handled as follows
\begin{equation}
\int_{Q_t}{\left|\nabla \chi \right|^2} + \int_{Q_t}{\xi \chi} + \int_{Q_t}{ [\sigma'(\chi_1) - \sigma'(\chi_2) ] \chi} \geq c \left\| \chi \right\|^2_{L^2(0,T;V)}.
\end{equation}

For the other terms we apply the same procedure seen in Section \ref{exist} and we get
\begin{gather}
\int_{Q_t}{\zeta \vartheta}
+ \left\| 1 \ast \vartheta \right\|^2_{L^{\infty}(0,T;V)} + \left\|  \chi \right\|^2_{L^2(0,T;V)}  \nonumber \\
\leq \overline{M} \left( \left\| \eta_0 \right\|^2_H + \left\|g\right\|^2_{L^2(0,T;H)} + \left\|h\right\|^2_{L^2(0,T;L^2(\Gamma))} \right),
\end{gather}
where $\overline{M} = \overline{M}(\Omega, T)$.

Finally, we can notice that, since $\Psi$ is strictly convex, its subdifferential $\partial \Psi$ is strictly monotone, then we have $\vartheta_1 = \vartheta_2$ and $\chi_1=\chi_2$, if we set $g_1 = g_2$, $h_1 = h_2$ and $\eta_{0,1} = \eta_{0,2}$.

\section{Acknowledgments}
The author gratefully acknowledges Prof. Elisabetta Rocca for proposing this problem and for her remarkable help in solving it.

\end{document}